\documentclass{article}
\usepackage{amssymb}
\usepackage{amsfonts}
\usepackage{amsmath}
\usepackage{mathrsfs}
\usepackage{graphicx}
\usepackage{amsbsy}
\usepackage{theorem}
\usepackage{color}
\usepackage[colorlinks=true]{hyperref}
\usepackage{dsfont}
\usepackage[titletoc,toc,title]{appendix}
\newtheorem{theorem}{Theorem}[section]
\newtheorem{lemma}[theorem]{Lemma}
\newtheorem{corollary}[theorem]{Corollary}
\newtheorem{definition}[theorem]{Definition}
\newtheorem{proposition}[theorem]{Proposition}
\newtheorem{remark}[theorem]{Remark}

\def\beq{\begin{equation}\displaystyle}
\def\eeq{\end{equation}}
\def\bel{\begin{equation} \displaystyle \begin{array}{l} }
\def\eel{\end{array} \end{equation} }
\def\bell{\begin{equation} \displaystyle \begin{array}{ll}  }
\def\eell{\end{array} \end{equation} }

\def\bea{\begin{eqnarray}}
\def\eea{\end{eqnarray} }
\def\bean{\begin{eqnarray*}}
\def\eean{\end{eqnarray*} }
\newenvironment{proof}{\noindent{\bf Proof.~}}
{{\mbox{}\hfill {\small \fbox{}}\\}}

\def\RR{\mathbb{R}}

\newcommand{\bepa}{\left\{ \begin{array}{l}}
\newcommand{\eepa}{\end{array} \right.}
\newcommand{\p}{\partial}
\newcommand{\f}{\frac}

\def\ds{\displaystyle}

\def\eps{\varepsilon}

\def\pa{\partial}

\def\calB{{\mathcal B}}

\newcommand{\usar}{\underline{u}}
\newcommand{\ubar}{\overline{u}}
\newcommand{\calL}{\mathcal{L}}

\begin{document}

\title{Hindrances to bistable front propagation: application to \textit{Wolbachia} invasion}
\author{Gr\'{e}goire Nadin\thanks{LJLL, UPMC, 5 place Jussieu, 75005 Paris France} \and Martin Strugarek\thanks{AgroParisTech, 16 rue Claude Bernard, F-75231 Paris Cedex 05 
               \& LJLL, UPMC, 5 place Jussieu, 75005 Paris France, \href{mailto:strugarek@ljll.math.upmc.fr}{strugarek@ljll.math.upmc.fr}}  
               \and Nicolas Vauchelet\thanks{LAGA - UMR 7539
 Institut Galil\'{e}e
 Universit\'{e} Paris 13
 99, avenue Jean-Baptiste Cl\'{e}ment
 93430 Villetaneuse - France}}

\maketitle

\begin{abstract}
We study the biological situation when an invading population propagates and replaces an existing population with different characteristics.
For instance, this may occur in the presence of a vertically transmitted infection causing a cytoplasmic effect similar to the Allee effect ({\it e.g.} {\it Wolbachia} in {\it Aedes} mosquitoes): the invading dynamics we model is bistable.

After quantification of the propagules, a second question of major interest is the invasive power. How far can such an invading front go, and what can stop it?
We rigorously show that a heterogeneous environment inducing a strong enough population gradient can stop an invading front, which will converge in this case to a stable front. We characterize the critical population jump, and also prove the existence of unstable fronts above the stable (blocking) fronts. Being above the maximal unstable front enables an invading front to clear the obstacle and propagate further.

We are particularly interested in the case of artificial {\it Wolbachia} infection, used as a tool to fight arboviruses.

\end{abstract}

\section{Introduction}

The fight against world-wide plague of dengue (see \cite{Bha.Burden}) and of other arboviruses has motivated extensive work among the scientific community. 
Investigation of innovative vector-control techniques has become a well-developed area of research.
Among them, the use of {\it Wolbachia} in {\it Aedes} mosquitoes to control diseases (see \cite{Wal.wMel,Alp.Aedes}) has received considerable attention. This endo-symbiotic bacterium is transmitted from mother to offspring, it induces cytoplasmic incompatibility (crossings between infected males and uninfected females are unfertile) and blocks virus replication in the mosquito's body. Artificial infection can be performed in the lab, and vertical transmission allows quick and massive rearing of an infected colony.
Pioneer mathematical modeling works on this technique include \cite{BT,HugBri.Modelling,Han.Strategies}.

We are mostly interested in the way space interferes during the vector-control processes.
More precisely, we would like to understand when mathematical models including space can effectively predict the blocking of an on-going biological invasion, which may have been caused, for example, by releases during a vector-control program.

The observation of biological invasions, and of their blocking, has a long and rich history. We simply give an example connected with {\it Wolbachia}.
In the experimental work \cite{stableCoexistence}, it was proved that a stable coexistence of several (three) natural strains of {\it Wolbachia} can exist, in a {\it Culex pipiens} population. The authors mentioned several hypotheses to explain this stability.
Our findings in the present paper - using a very simplified mathematical model - partly supports the hypotheses analysis conducted in the cited article. Namely, ``differential adaptation'' cannot explain the blocking, while a large enough ``population gradient'' can, and we are able to quantify the strength of this gradient, potentially helping validating or discarding this hypothesis.

Although the field experiments have not yet been conducted for a significantly long period, artificial releases of {\it Wolbachia}-infected mosquitoes (see~\cite{Hof.Successful,Ngu.Field}) also seem to experience such ``stable fronts'' or blocking phenomena (see \cite{Hof.Stability,Yea.Dynamics}).
This issue was studied from a modeling point of view in~\cite{BT,chankim} (reaction-diffusion models), \cite{Han.Modelling} (heterogeneity in the habitat) and \cite{Han.Density} (density-dependent effects slowing the invasion), among others.

\bigskip

In order to represent a biological invasion in mathematical terms as simply as possible, reaction-diffusion equation have been introduced (for the first time in \cite{Fisher1937} and \cite{KPP1937}) in the form
\beq\label{eq:u}
\pa_t u - \Delta u = f(u),
\eeq
where $t \geq 0$ and $x \in \RR^d$ are respectively time and space variables, and $u(t,x)$ is a density of alleles in a population, at time $t$ and location $x$.
This very common model to study propagation across space in population dynamics enhances a celebrated and useful feature: existence (under some assumptions on $f$) of traveling wave solutions.
In space dimension $1$, a traveling wave is a solution $u(t, x) = \widetilde{u} (x - c t)$ to \eqref{eq:u}, where $c \in \RR$, $\widetilde{u}$ is a monotone function
$\RR \to \RR$, and $\widetilde{u} (\pm \infty) \in f^{-1} (0)$.
By convention we will always use decreasing traveling waves. 
They have a constant shape and move at the constant speed $c$.

The quantity $u$ may represent the frequency of a given trait (phenotype, genotype, behavior, infection, etc.) in a population.
In this case, the model below has been introduced in order to account for the effect of spatial variations in the total population
density $N$
(see \cite{Barton,BT}) in the dynamics of a frequency~$p$
\begin{equation}\label{eq:pN}
\pa_t p - \Delta p - 2 \frac{\nabla N\cdot \nabla p}{N} =f(p).
\end{equation}

In some cases, the total population density $N$ may be affected by the trait frequency $p$, and even depend explicitly on it.
In the large population asymptotic for the spread of {\it Wolbachia}, where $p$ stands for the infection frequency, it was proved (in \cite{reduction})
that there exists a function $h : [0, 1] \to (0, +\infty)$ such that $N = h(p) + o(1)$, in the limit when population size and reproduction rate are large and of same order.

Hence we can write the first-order approximation 
\beq\label{eq:p}
\pa_t p - \Delta p - 2 \frac{h'(p)}{h(p)}|\nabla p|^2 =f(p).
\eeq

Our main results are the characterization of the asymptotic behavior of $p$ in two settings: for equation \eqref{eq:pN} when $N$ only depends on
$x$, and for equation \eqref{eq:p} in all generality. Both of them may be seen as special cases of the general problem
\[
 \p_t p - \Delta p - 2 \nabla \big( V\big(x, p(t, x) \big) \big) \cdot \nabla p = f(p).
\]
For \eqref{eq:pN} with $d=1$, our characterization is sharp when $\p_x \log N$ is equal to a constant times the characteristic function of an interval.
Overall, two possible sets of asymptotic behaviors appear. 
On the first hand, the equation can exhibit a sharp threshold property, dividing the initial data between those leading to invasion
of the infection ($p \to 1$) and those leading to extinction ($p \to 0$) as time goes to infinity. In this case, the threshold is constituted by initial data
leading convergence to a ground state (positive non-constant stationary solution, going to $0$ at infinity). It is a sharp threshold, which implies that the ground state is
unstable.
We show that such a threshold property always holds for equation \eqref{eq:p}, and occurs in some cases for equation \eqref{eq:pN}. 
On the other hand, the infection propagation can be blocked by what we call here a ``barrier'' that is a stationary solution or, in the biological context, a blocked propagation front. 
We show that this happens in \eqref{eq:pN}, essentially when $\p_x \log N$ is large enough. 
This asymptotic behavior differs from convergence towards a
ground state in the homogeneous case. 
Indeed, even though the solution converges towards a positive stationary solution, we prove that in this barrier case, the blocking is actually stable. 
Some crucial implications for practical purposes (use of {\it Wolbachia} in the field) of this stable failure of
infection propagation are discussed.

From the mathematical point of view, our work on \eqref{eq:pN} makes use of a phase-plane method that can be found in \cite{lewiskeener} (and also in  \cite{chapuisat} and \cite{polacik}) to study similar problems.
It helps getting a good intuition of the results, coupled with a double-shooting argument. We note that a shooting method was also used in \cite{malaguti} for ignition-type nonlinearity, in a non-autonomous setting, to get similar results under monotonicity assumptions we do not require here.

The paper is organized as follows.
Main results on both \eqref{eq:p} and \eqref{eq:pN} are stated in Section \ref{sec:results}, where their biological meaning is explained.
We also give illustrative numerical simulations. After a brief recall of well-known facts on bistable reaction-diffusion in Section~\ref{sec:gen}, we prove our results on \eqref{eq:p} in Section~\ref{sec:infdep}, and on \eqref{eq:pN} in Section \ref{sec:het}.
Finally, Section \ref{sec:dis} is devoted to a discussion on our results, and on possible extensions.
Moreover, because it was the work that first attracted us to this topic, we expand in Section \ref{sec:loc} on the concept of local barrier
developed by Barton and Turelli in \cite{BT}, and relate it to the present article.

\section{Main results}
\label{sec:results}

\subsection{Statement of the results}

\subsubsection{Results on the infection-dependent case}

Our first set of results is concerned with \eqref{eq:p}, where the total population is a function of the infection frequency.

We notice that the problem \eqref{eq:p} is invariant by multiplying $h$ by any $\lambda \in \RR^*$.
Without loss of generality we therefore fix $\int_0^1 h^2(\xi)\,d\xi = 1$, and state
\begin{theorem}
Let $H$ be the antiderivative of $h^2$ which vanishes at $0$, that is
$H(x) := \int_0^x h^2(\xi)\,d\xi$.
$H$ is a $\mathcal{C}^1$ diffeomorphism from $[0,1]$ into $[0, 1]$.

Let $g : [0, 1 ] \to [0, 1]$ such that for all $x \in [0, 1]$, $g(H(x)) = f(x) h^2 (x)$.

There exists a traveling wave for \eqref{eq:p} if and only if there exists a traveling wave for \eqref{eq:u} with reaction term $g$ ({\it i.e.} $\p_t u - \p_{xx} u = g(u)$).
In addition:
\begin{enumerate}
\item If $f$ satisfies the KPP (named after \cite{KPP1937}) condition $f(x) \leq f'(0) x$ and if $H$ is concave (which is equivalent to $h' \leq 0$), then there
exists a minimal wave speed $c_* := 2 \sqrt{g'(0)}$ for traveling wave solutions to \eqref{eq:p}. This means that for all $c \geq
c_*$, there exists a unique traveling wave solution to \eqref{eq:p} with speed $c$.

\item If $f$ is bistable then there exists a unique traveling wave for \eqref{eq:p}. Its speed has the sign of \[\int_0^1 f(x) h^4
(x) dx.\]
\end{enumerate}
\label{thm:mainInf}
\end{theorem}
Depending on the initial data, in this case, solution can converge to $1$ (``invasion''), initiating a traveling wave with positive speed, or to $0$ (``extinction'').
Note that non-constant $h$ may have a huge impact in the asymptotic behavior, possibly reversing the traveling wave speed: in this case, $0$ would become the invading state instead of $1$. 

In the case of {\it Wolbachia}, we discuss the expression of $h$ in Subsection \ref{subs:h}, and give a numerical example of this situation in Subsection \ref{subs:numerics}.

We can construct a family of compactly supported ``propagules'', that is functions which ensure invasion. 
\begin{proposition}
For all $\alpha \in (\theta_c, 1)$, there exists $v_{\alpha} \in \mathcal{C}^2_p (\RR, [0, \alpha])$ ($v_{\alpha}$ is continuous
and of class $\mathcal{C}^2$ by parts on $\RR$), whose support is equal to $[-L_{\alpha}, L_{\alpha}]$ for a known $L_{\alpha} \in
(0, +\infty)$ (given below by \eqref{eq:Lalpha}), such that $0 \leq v_{\alpha} \leq \alpha$, $\max v_{\alpha} = v_{\alpha} ( 0) = \alpha$,
$v_{\alpha}$ is symmetric and radial-non-increasing, and $v_{\alpha}$ is a sub-solution to \eqref{eq:p}.
\label{prop:propagules}
\end{proposition}
We name $v_{\alpha}$ $\alpha$-bubble (associated with \eqref{eq:p}), or $\alpha$-propagule, following the definition in \cite{BT}.

\subsubsection{Results on the heterogeneous case}
Our second set of results deals with the situation where the total population of mosquitoes strongly increases in a 
given region of the domain. In this case, the total population $N$ is given and we consider
the model \eqref{eq:pN}. 
Before stating our main result on equation \eqref{eq:pN}, we introduce the concept of {\it propagation barrier} (which we will simply call {\it barrier} below).

To fix the ideas and get a tractable problem, we assume that $N$ increases (exponentially)
in a given region of spatial domain and is constant in the rest of the domain. 
We consider that the domain is one-dimensional and therefore investigate the differential equation
\beq
\pa_t p - \pa_{xx} p - 2 \pa_x(\log N) \pa_x p = f(p).
\label{eq:onR}
\eeq

In view of the setting we have in mind for $N$ we let, for some $C, L > 0$:
\beq
\pa_x \log (N) = \left\{\begin{array}{ll}
\displaystyle\f{C}{2}, \qquad & \mbox{ on } [-L,L], \\
0, \qquad & \mbox{ on } \RR\setminus[-L,L].
\end{array}\right.
\label{eq:logN}
\eeq

Existence of a stationary wave for this problem boils down to the existence of a solution to
\beq \label{eq:TW}
\left\{\begin{array}{ll}
-p'' -  C p' = f(p), \qquad &\mbox{ on } [-L,L],  \\
-p'' = f(p), \qquad & \mbox{ on } \RR\setminus[-L,L], \\
p(-\infty)=1, \quad  p(+\infty) = 0, \quad  p>0.
\end{array}\right.
\eeq

In the context of our study, stationary solutions to \eqref{eq:onR} with prescribed behavior at infinity, that is solutions of
\eqref{eq:TW}, play the role of {\it barriers}, blocking the propagation of the infection.
\begin{definition}
 We name a $(C, L)$-barrier any solution to \eqref{eq:TW}.
  For any bistable function $f$ we define the barrier set 
 \beq
  \calB (f) := \big\{ (C, L) \in (0, + \infty)^2, \text{ there exists a } (C, L)\text{-barrier} \big\}.
 \eeq
 \label{def:barriers}
\end{definition}

As we will recall in Section \ref{sec:gen}, in the bistable case there exists a unique (up to translations) traveling wave solution. This solution can be seen as a solution to the limit problem of \eqref{eq:TW} as $L \to +\infty$. We make this intuition more precise in this paper (see in particular Proposition \ref{prop:Lstarasympt} below).

The bistable traveling wave is associated with a unique speed that we denote $c_*(f)$ (see Section \ref{sec:gen} for definitions and a brief review of classical results on bistable reaction-diffusion).

\begin{theorem}
Let $C > 0, L > 0$ and assume $N$ is given by \eqref{eq:logN}.
For $C > c_*(f)$, there exists $L_* (C) \in (0, +\infty)$ such that $(C, L) \in \calB(f)$ if and only if $L \geq L_* (C)$.
\label{thm:mainHet}
\end{theorem}

Existence of a barrier, as stated in Theorem \ref{thm:mainHet}, has strong and direct consequences on the asymptotic behavior of solutions to \eqref{eq:pN}.
\begin{proposition}
Assume $N$ is defined by \eqref{eq:logN}. If $(C, L) \in \calB(f)$ we denote by $p_B$ a solution to the standing wave problem
\eqref{eq:TW}.
 Then any solution of \eqref{eq:onR}
with initial value $p^0$ satisfying $p^0 \leq p_B$ has stopped propagation, which means that $\forall x \in \RR, \limsup_{t \to
\infty} p(t, x) < 1$.
  More precisely, 
  \[
   \forall t \geq 0, \, p (t, x) \leq p_B (x).
  \]
 
On the contrary, assume that either \eqref{eq:TW} has no solution (i.e. $(C, L) \not\in \calB(f)$) and $\exists \lim_{- \infty} p^0
=1$, or there exists a solution $p_B$ to \eqref{eq:TW} which is unstable from above (in the sense of Definition \ref{def:belabo}),
such that $p_0 > p_B$ and there is no other solution $p_{B'}$ to \eqref{eq:TW} satisfying $p_{B'} > p_B$. In this case $p$
propagates, that is:
 \[
  \forall x \in \RR, \, \limsup_{t \to \infty} p(t, x) = 1.
 \]
 \label{prop:passblock}
\end{proposition}

We also characterize the barriers
\begin{proposition}
Let $(C, L) \in \calB(f)$. Then 
\begin{enumerate}
\item Any $(C, L)$-barrier ({\it i.e.} solution of \eqref{eq:TW}) is decreasing.
\item If $L > L_* (C)$ then there exists at least two $(C, L)$-barriers.
\item The $(C, L)$-barriers are totally ordered, hence we can define a maximal and a minimal element among them.
\item The maximal $(C,L)$-barrier is unstable from above and the minimal one is stable from below (in the sense of Definition \ref{def:belabo} below).
\end{enumerate}
\label{prop:barriers}
\end{proposition}

We also get a picture of the behavior of $L_* (C)$:
\begin{proposition}
The function $L_*$ is decreasing and satisfies
\[
	\lim_{C \to c_*(f)^+} L_*(C) = +\infty, \quad L_*(C) \sim \f{1}{4 C} \log \big( 1 - \f{F(1)}{F(\theta)} \big) \text{ when } C \to +\infty.
\]
\label{prop:Lstarasympt}
\end{proposition}

Instead of restricting to a constant (logarithmic) population gradient, we can very well let it vary freely.
To do so we introduce a set of gradient profiles which we denote by $X$. For example,
\beq
X := \{ h : \RR \to \RR_+, \, h \in L^{\infty} \text{ with compact support.} \}
\label{set:X}
\eeq
Then, the barriers may be defined in a similar fashion as before.
\begin{definition}
 For $h \in X$, a $h$-barrier is any solution to the ``standing wave equation''
 \beq
 \bepa
 - p'' - h(x) p' = f(p) \text{ on } \RR,
 \\[10pt]
 p(-\infty) = 1, \quad p(+\infty) = 0.
 \eepa
 \label{eq:GSW}
 \eeq
 
 We define the barrier set associated with \eqref{set:X}
 \[
  \calB_X (f) := \{ h \in X, \text{ there exists a } h\text{-barrier} \}.
 \]
\end{definition}

In this setting, a meaningful extension of Theorem \ref{thm:mainHet} is the following
\begin{corollary}
 Let $h \in X$. If $(C, L) \in \calB(f)$ and $h \geq C \chi_{[-L, L]}$ then $h \in \calB_X (f)$.
  Conversely, if $(C, L) \not\in \calB(f)$ and $h \leq C \chi_{[-L, L]}$ then $h \not\in \calB_X (f)$.
\label{cor:nonconstantgradient}
\end{corollary}

\subsection{Biological interpretation}

Our results on possible propagation failures can be summarized and interpreted easily.

On the first hand, if the size of the population is regulated only by the level of the infection (or the trait frequency), then in a homogeneous medium no stable blocked front can appear (this is the sharp threshold property implied by Theorem \ref{thm:mainInf}),
except in the very particular case when $\int_0^1 f(x) h^4 (x) dx = 0$. 
This situation can be understood as the limit when local demographic equilibrium is reached much faster than the infection process (or when the population is typically large, as in the asymptotic from \cite{reduction}), which makes sense in the context of {\it Wolbachia} because the infection is vertically transmitted.

On the second hand, if the carrying capacity (or ``nominal population size'') is heterogeneous (in space), then an increase in the population size raises a hindrance to propagation, that can be sufficient to effectively block an invading front (Theorem \ref{thm:mainHet}), and give rise to a stable transition area (as observed in~\cite{stableCoexistence}), even if the infection status does not modify the individuals' fitness.
This situation is particularly adapted to a wide range of {\it Wolbachia} infections, when several natural or artificial strains do not have very different impacts on the host's fitness.
We note that the case when the heterogeneity concerns the diffusivity rather than the population size was treated in \cite{lewiskeener}, yielding the same conclusion: a large-enough area of low-enough diffusivity stops the propagation.

From our results, we draw two conclusions that are relevant in the context of biological invasions.

First, fitness cost (and cytoplasmic incompatibility level, in the case of {\it Wolbachia}) determines the existence of an invading front in a homogeneous setting, and eventually its speed.  However, ecological heterogeneity (rather than fitness cost) seems to play a prominent role in propagation failure - or success - of a given infection.

Second, the existence of a stable (from below) front implies the existence of an unstable (from above) one, as stated in Proposition \ref{prop:barriers}. Therefore, any of the heterogeneity-induced hindrances to propagation that have been identified (here and in \cite{lewiskeener}) can be jumped upon. It suffices that the infection wave reaches the unstable front level. 
Computing the location and level of this theoretical ``unstable front'', in the presence of an actual ``stable front'', is extremely useful: either to estimate the risk that the infection propagates through the barrier into the sound area, or to know the cost of the supplementary introduction to be performed in order to propagate the infection through the obstacle (in the case of blocked propagation following artificial releases of {\it Wolbachia}, for example, as seems to be the case in the experimental situation described in \cite{Hof.Stability}).

\subsection{Numerical illustration}
\label{subs:numerics}

\begin{figure}[h!] 
 \includegraphics[width=\textwidth]{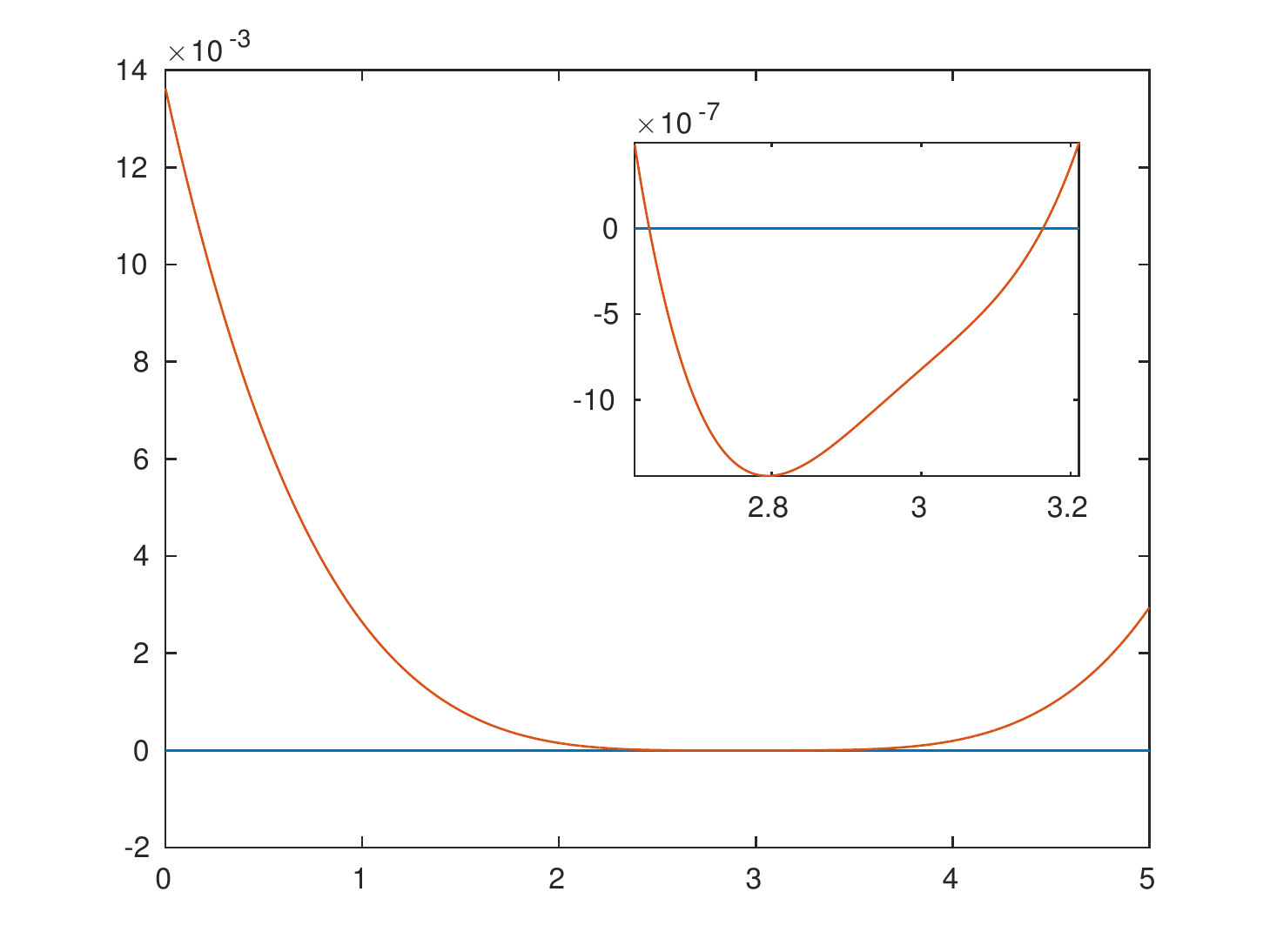}
 \caption{The function $\epsilon \mapsto \int_0^1 f(p) h_{\epsilon}^4 (p) dp$, whose sign is equal to that of the bistable traveling wave speed. The top-right angle plot is a zoom in the region where this sign is negative.}
 \label{fig:1}
\end{figure}

Figure \ref{fig:1} is an illustration of Theorem \ref{thm:mainInf}.
We choose $f$ and $h$ from the case of {\it Wolbachia} (see discussion on $h$ in Subsection \ref{subs:h}) with perfect vertical transmission and biological parameters selected after the choices in~\cite{reduction}:
\begin{align*}
 f(p) &= d_s p \frac{-s_h \delta p^2 + \big( \delta(1+s_h) - (1-s_f)\big) p + (1-s_f) - \delta}{s_h p^2 - (s_f + s_h)p + 1},
 \\
 h_{\epsilon}(p) &= 1 - \epsilon\frac{d_u}{\sigma F_u} \frac{(\delta-1)p+1}
{s_h p^2-(s_f+s_h)p +1}.
\end{align*}
We stick to this choice of $f$ for the other figures of this paper.
\begin{figure}[h!] 
 \includegraphics[width=.5\textwidth]{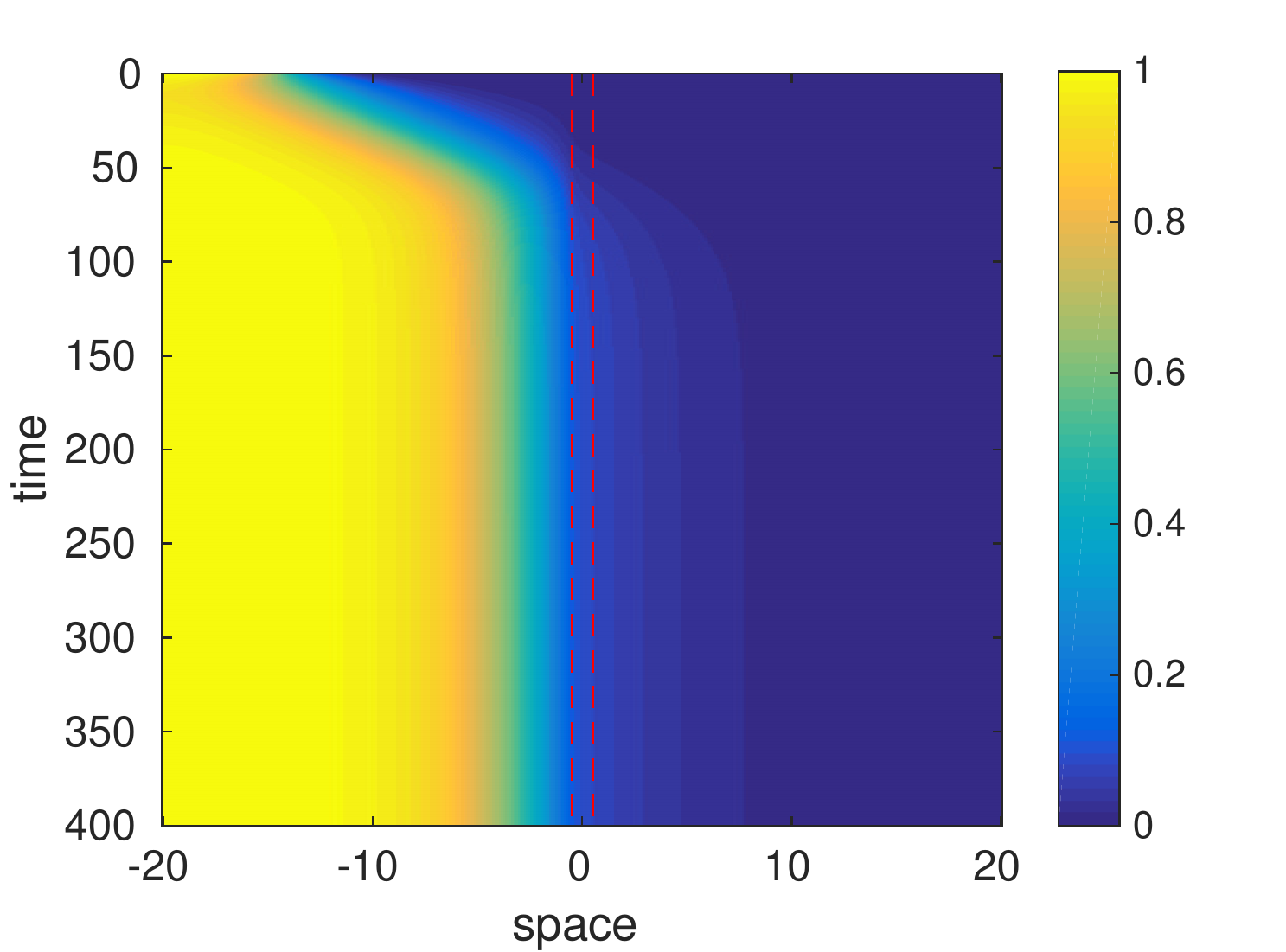}
 \includegraphics[width=.5\textwidth]{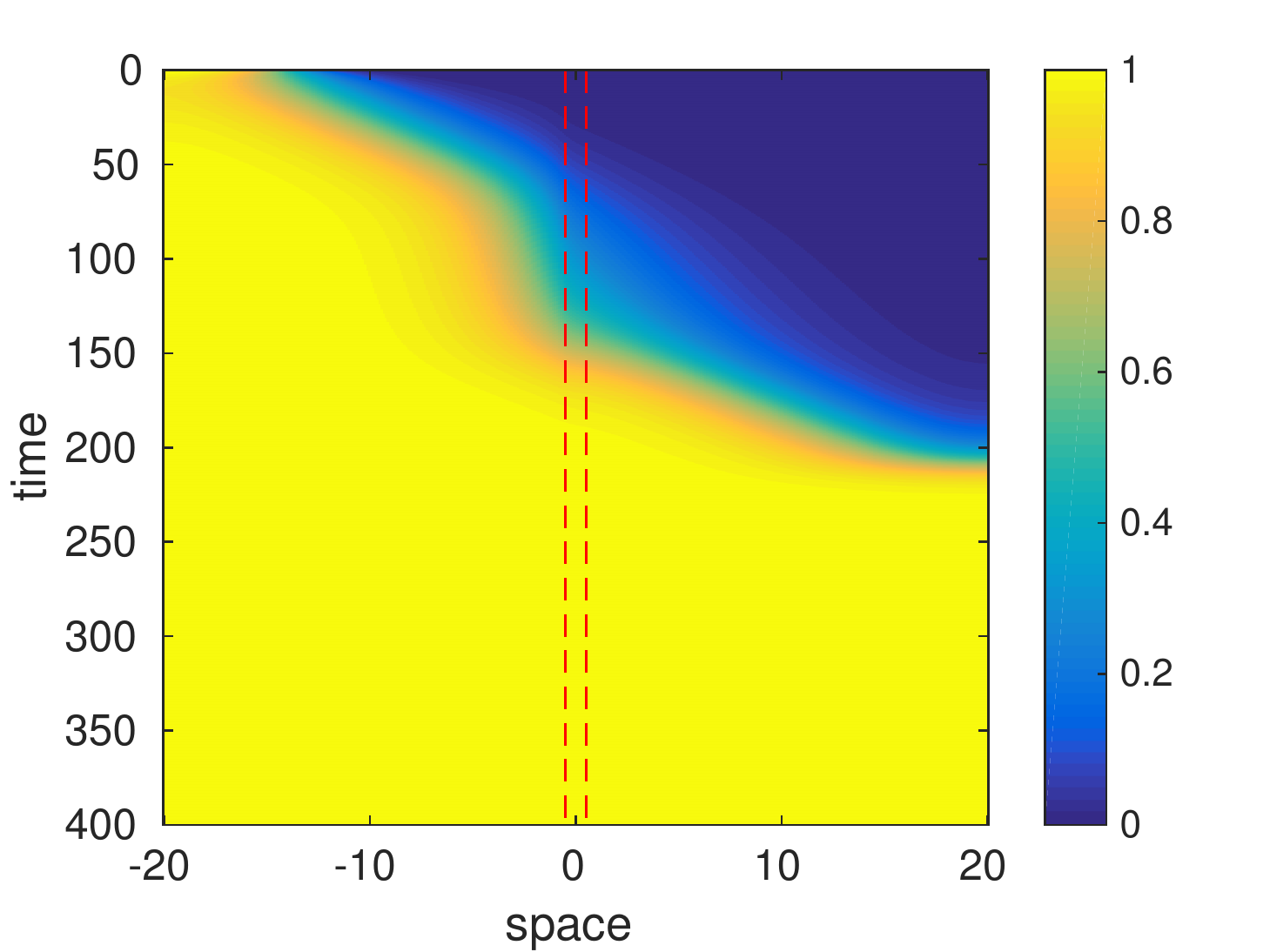}
 \caption{Plot of the proportion of the invading population with respect to time (y-axis) and space (x-axis). Two different population gradients are used with the same front-like initial data. The vertical red dotted lines mark the region $[-L, L]$ where the spatial gradient is applied. {\it Left}: Blocking with $L=0.5$ and $C=2$. {\it Right}: Propagation with $L=0.5$ and $C=1$.}
 \label{fig:2}
\end{figure}

\begin{figure}[h!] 
 \includegraphics[width=.5\textwidth]{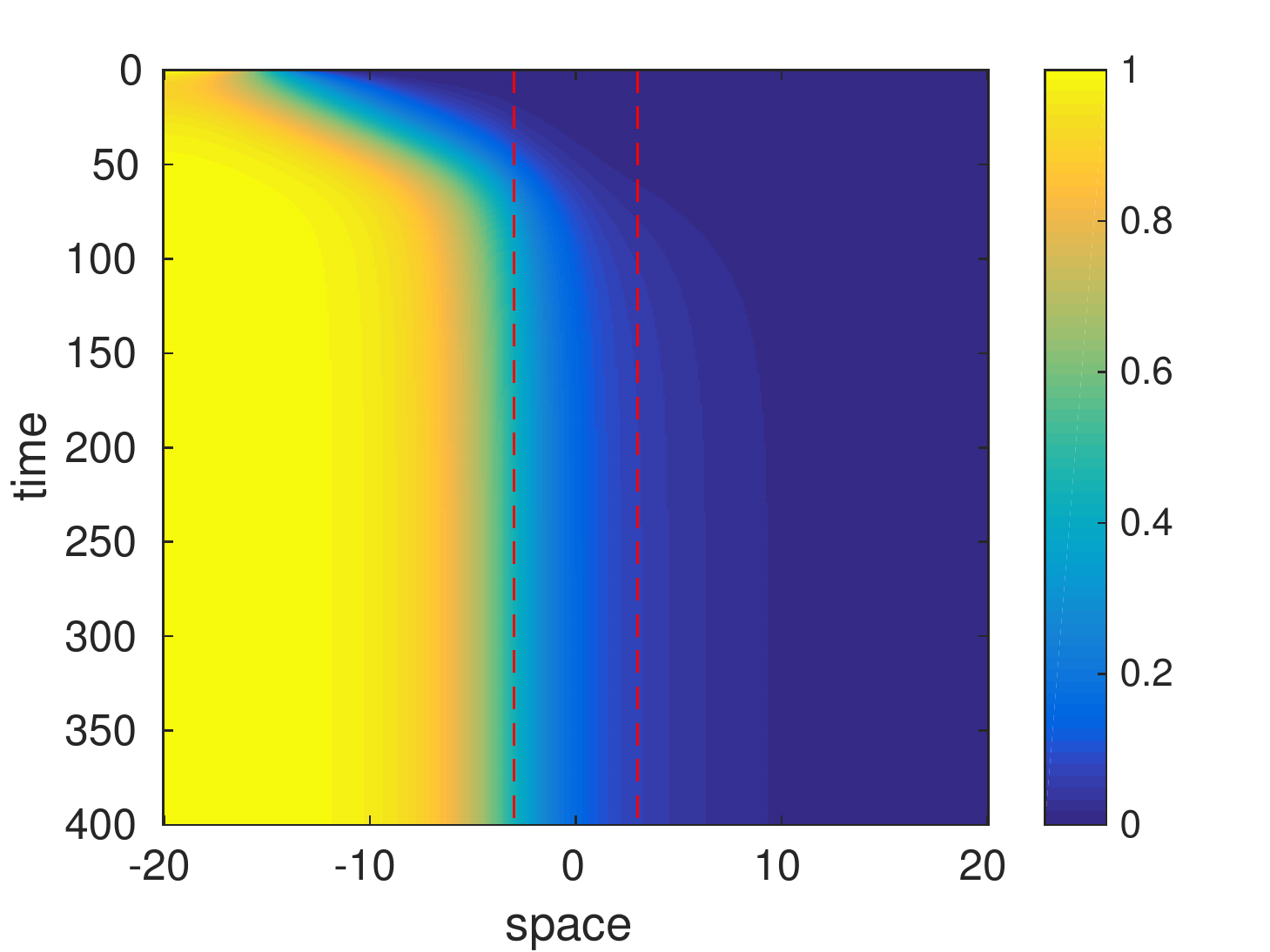}
 \includegraphics[width=.5\textwidth]{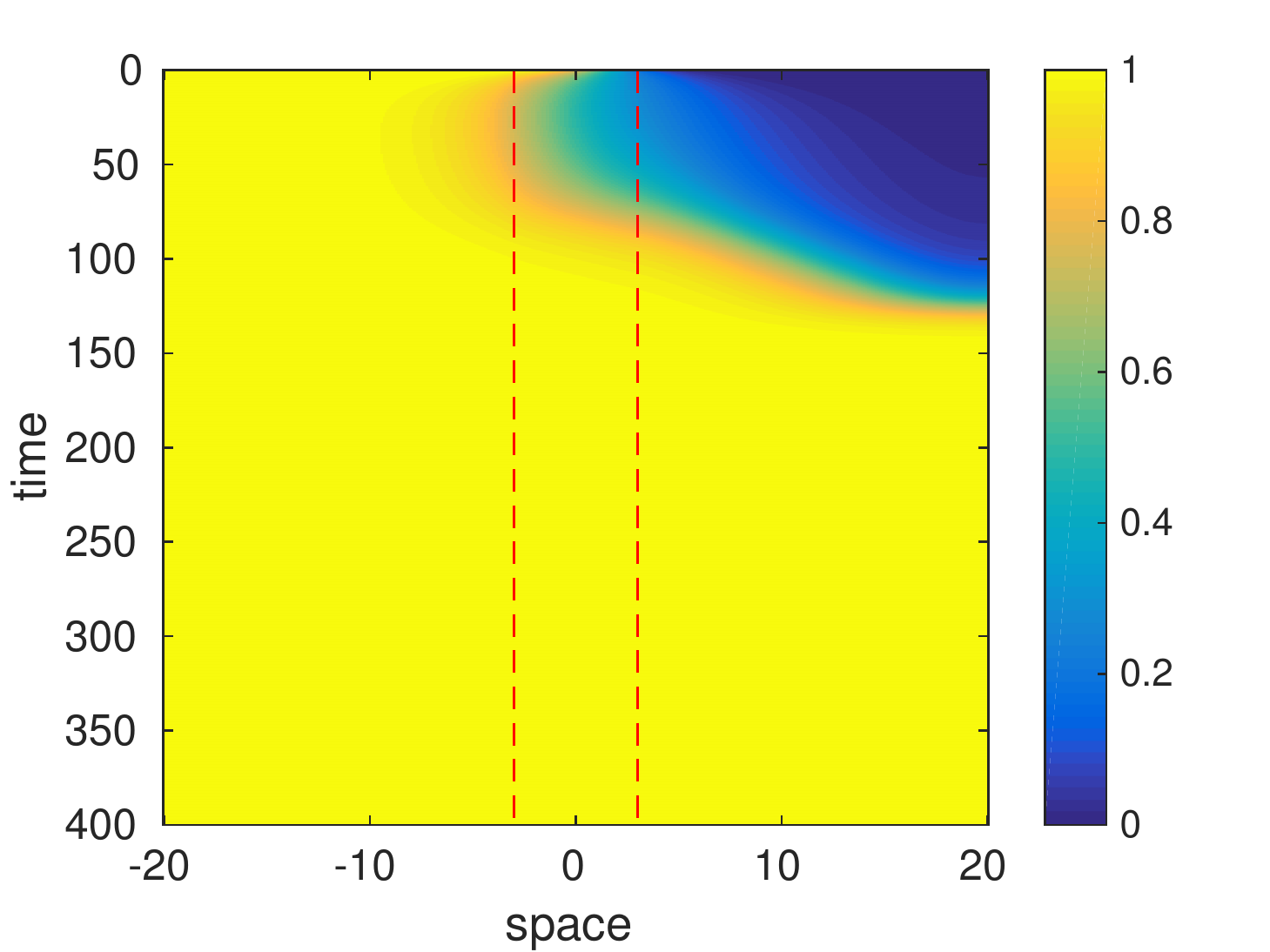}
 \caption{Plot of the proportion of the invading population with respect to time (y-axis) and space (x-axis). Two different front-like initial data are used with the same population gradient, $L = 3$ and $C=0.35$. The vertical red dotted lines mark the region $[-L, L]$ where the spatial gradient is applied. {\it Left}: Blocking with a Heaviside initial datum located at $-15$. {\it Right}: Propagation with a Heaviside initial datum located at $2$.}
 \label{fig:3}
\end{figure}

Figures \ref{fig:2}, \ref{fig:3} and \ref{fig:4} must be interpreted as follows: the $y$-axis, oriented to the bottom, is time $t \in [0, 400]$, while the $x$-axis is the space, $x \in [-20, 20]$.
The value of $p(t, x) \in [0, 1]$ is represented by a color, with the legend on the right-side of the plots.
Simulations were done using a centered finite-difference scheme for diffusion and Euler implicit for time, with discretization steps $\Delta t= 0.05 $ in time and $\Delta x = 0.1$ in space.
Vertical dotted red lines mark the spatial range (=support) of the population gradient.

\begin{figure}[h!] 
 \includegraphics[width=.5\textwidth]{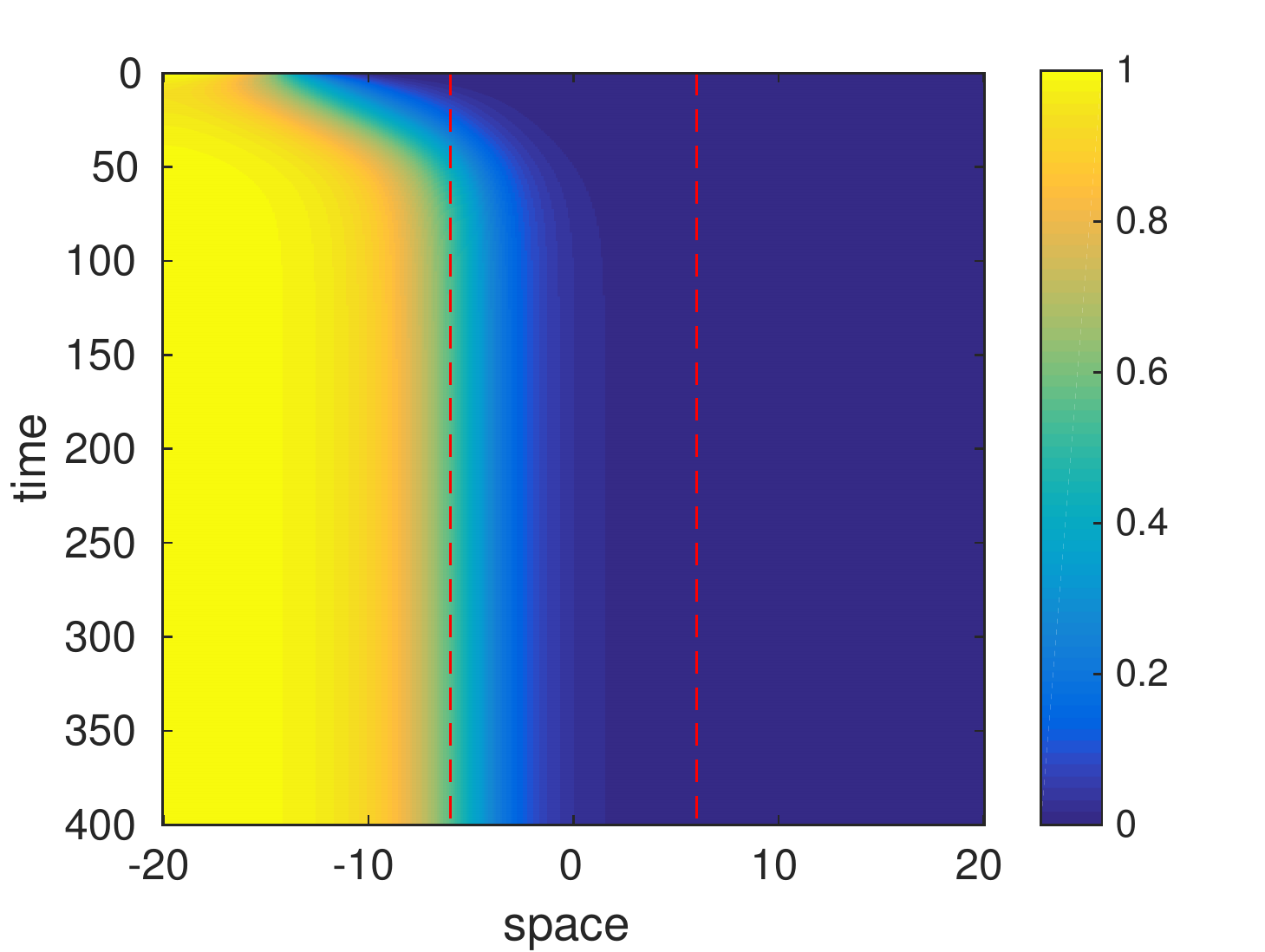}
 \includegraphics[width=.5\textwidth]{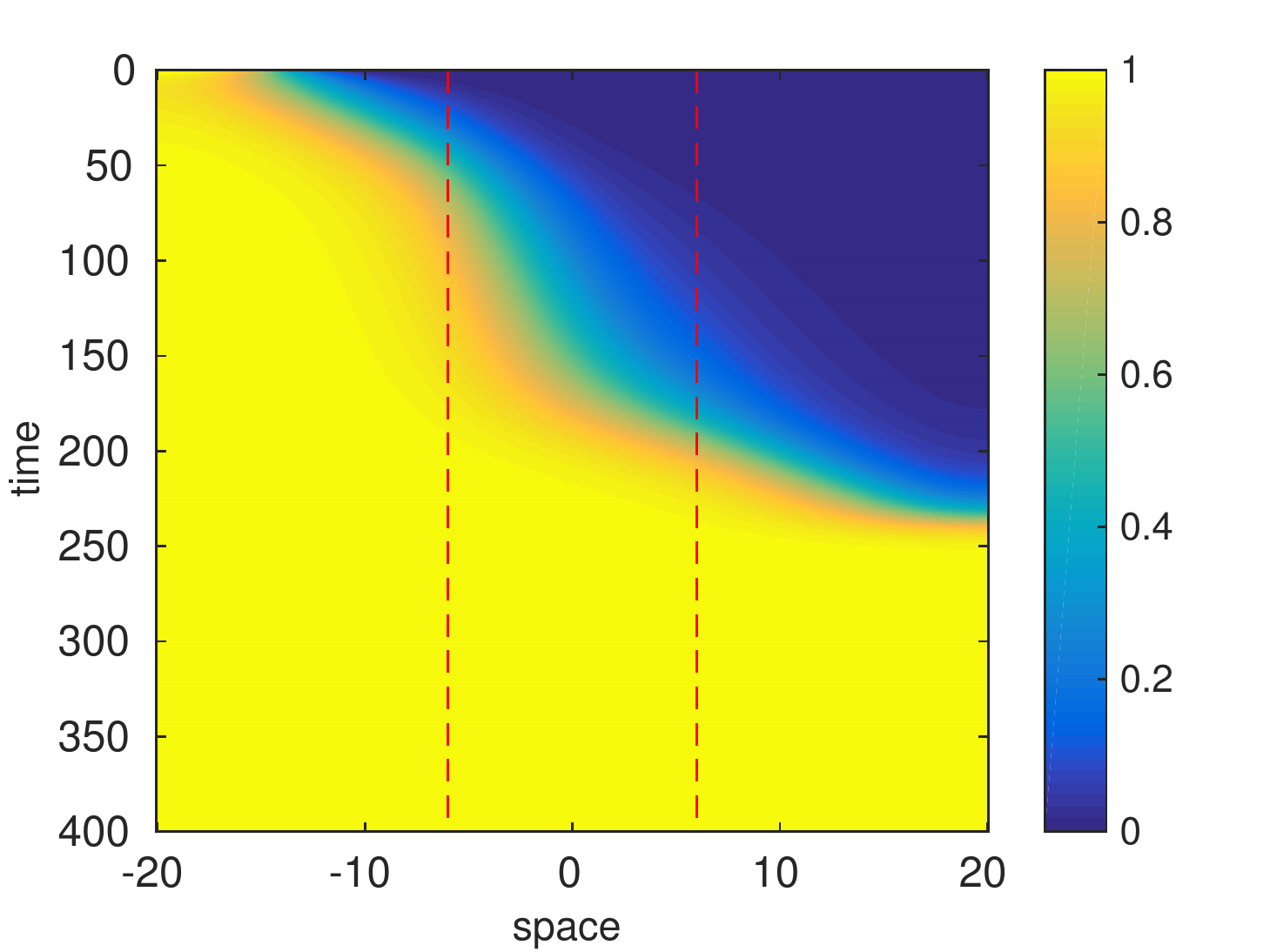}
 \caption{Plot of the proportion of the invading population with respect to time (y-axis) and space (x-axis). Two different, nontrivial population gradients ($h(x) = 4 C (x-L) (x+L)/L^2$) are used, with the same front-like initial data. The vertical red dotted lines mark the region $[-L, L]$ where the spatial gradient is applied. {\it Left}: Blocking with $L=6$, $C=0.5$. {\it Right}: Propagation with $L=6$, $C=0.2$.}
 \label{fig:4}
\end{figure}

\begin{figure}[h!]
\centering
 \includegraphics[width=.6\textwidth]{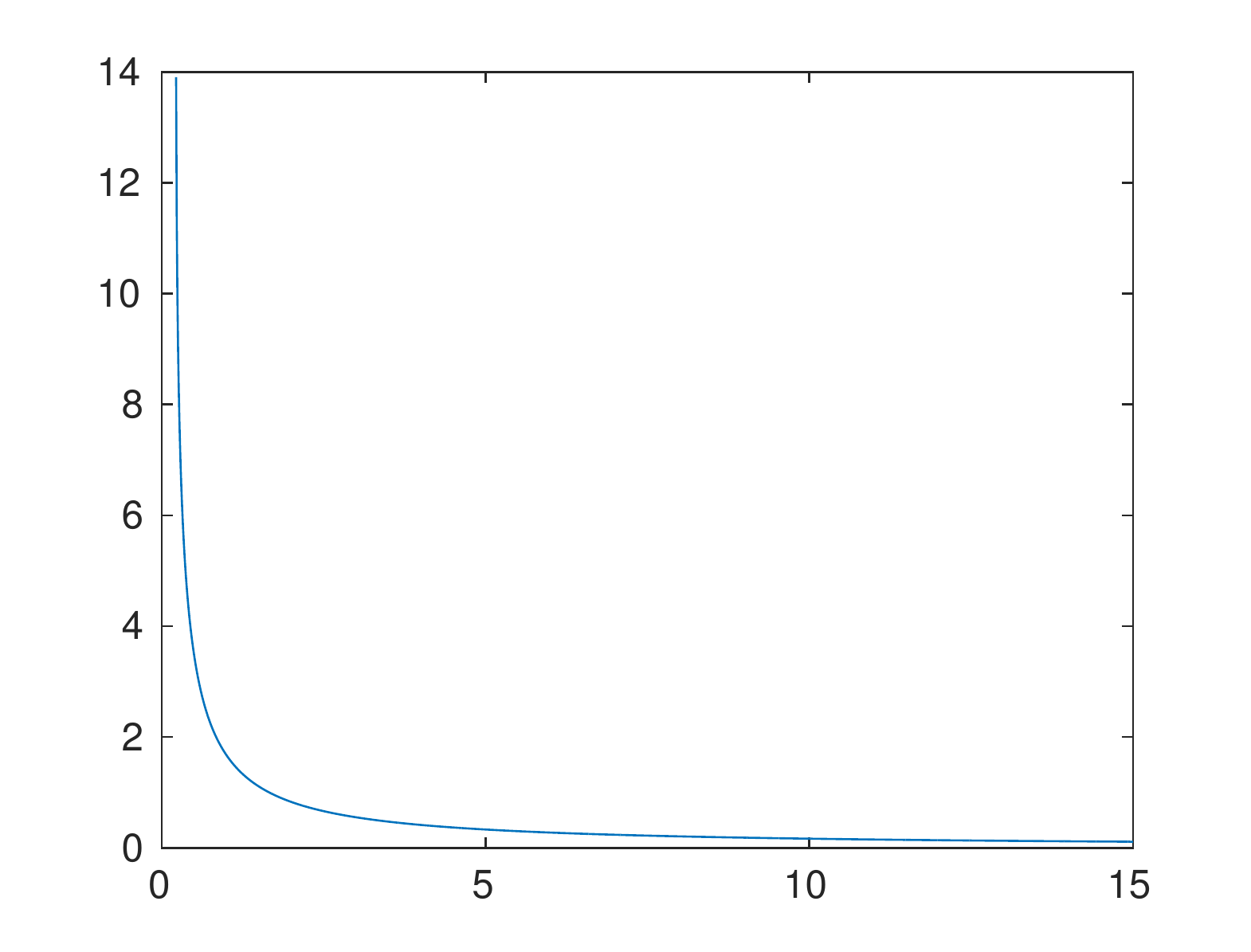}
 \caption{The minimal interval length $C \mapsto L_*(C)$ for which a logarithmic gradient constant equal to $C$ is sufficient to block invasion.}
 \label{fig:5}
\end{figure}

Figures \ref{fig:2} and \ref{fig:3} are illustrations of Proposition \ref{prop:passblock}.
On Figure \ref{fig:2}, the two plots differ only by the value of the population gradient $C$ (respectively equal to $2$ and $1$), imposed in both cases on the interval $[-0.5, 0.5]$.
The initial data is front-like, {\it i.e.} equal to $1$ on $[-20, -14]$.
On Figure \ref{fig:3}, the population gradient is fixed at $C=0.35$ with $L = 3$. The two plots differ by their initial data: they are still front-like, but on $[-20, -15]$ on the left-hand side, and on $[-20, 2]$ on the right-hand side.
On Figure \ref{fig:2}, on the left-hand plot we notice that a wave forms and propagates at a constant speed before being blocked, giving rise to a stable front ; while on the right-hand plot, the propagation occurs, and its speed is perturbed first by the heterogeneity, and then by the boundary of the discretization domain.
The interpretation is similar for Figure \ref{fig:3}.

Then, Figure \ref{fig:4} is an illustration of Corollary \ref{cor:nonconstantgradient}: it reproduces the behavior shown in Figure \ref{fig:2} for more sophisticated population gradients. We choose $h(x) = 4 C (x-L)(x+L)/L^2$, with $L=6$ and respectively $C=0.5$ (left-hand side) and $C=0.2$ (right-hand side), yielding blocking or propagation.

Finally Figures \ref{fig:5} and \ref{fig:6} illustrate Proposition \ref{prop:Lstarasympt}. Because of the high convergence speed of $C L_*(C)$ towards its finite limit for large $C$, we draw its logarithm in Figure \ref{fig:6} to get a better picture of convergence order.

We also note on Figure \ref{fig:6} that $C \mapsto C L_* (C)$ appears to be decreasing. 
We were only able to prove this fact asymptotically (as $C \to \infty$) and we refer to \cite{strug:these} for the explicit computations.

\begin{figure}[h!]
\includegraphics[width=.5\textwidth]{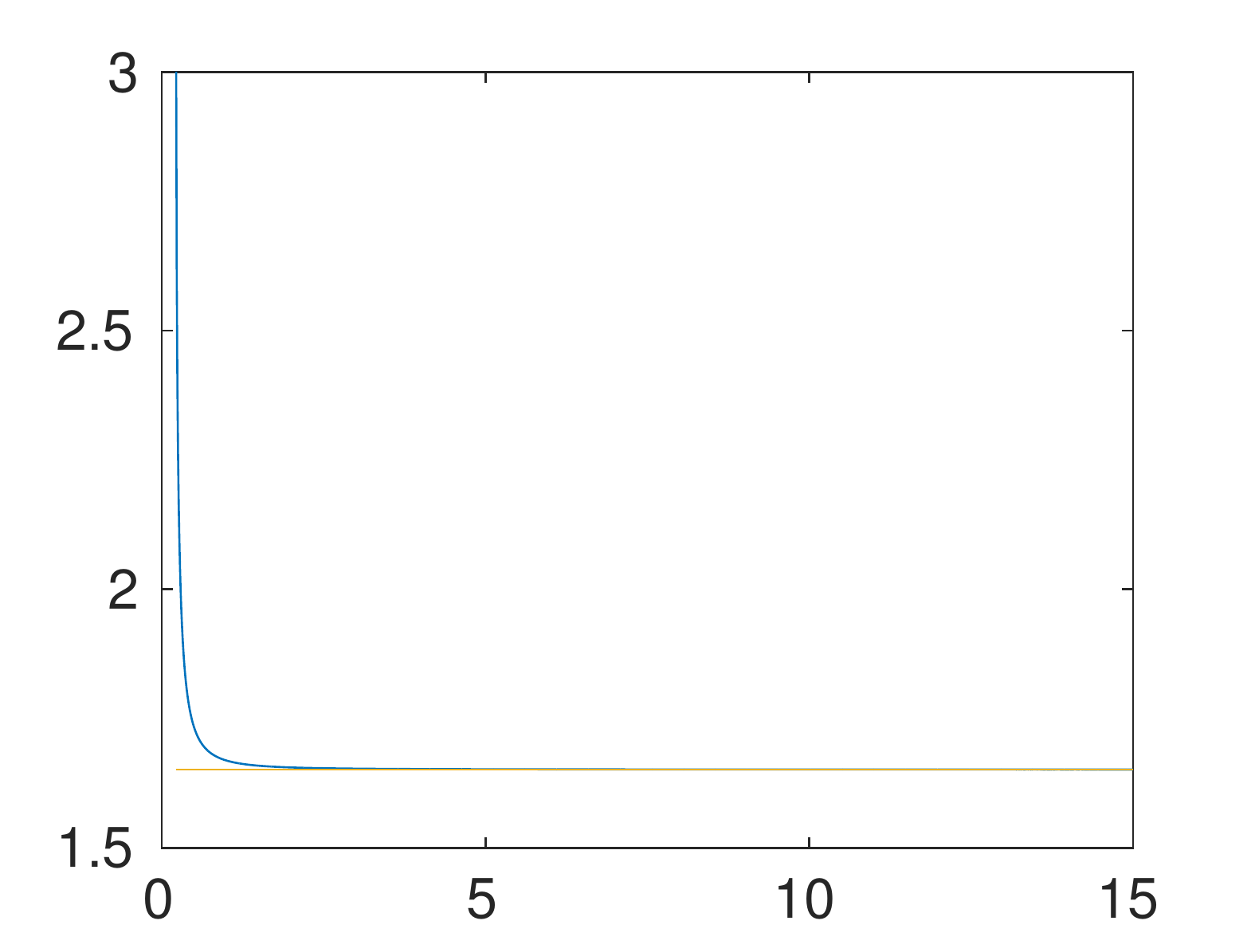}
\includegraphics[width=.5\textwidth]{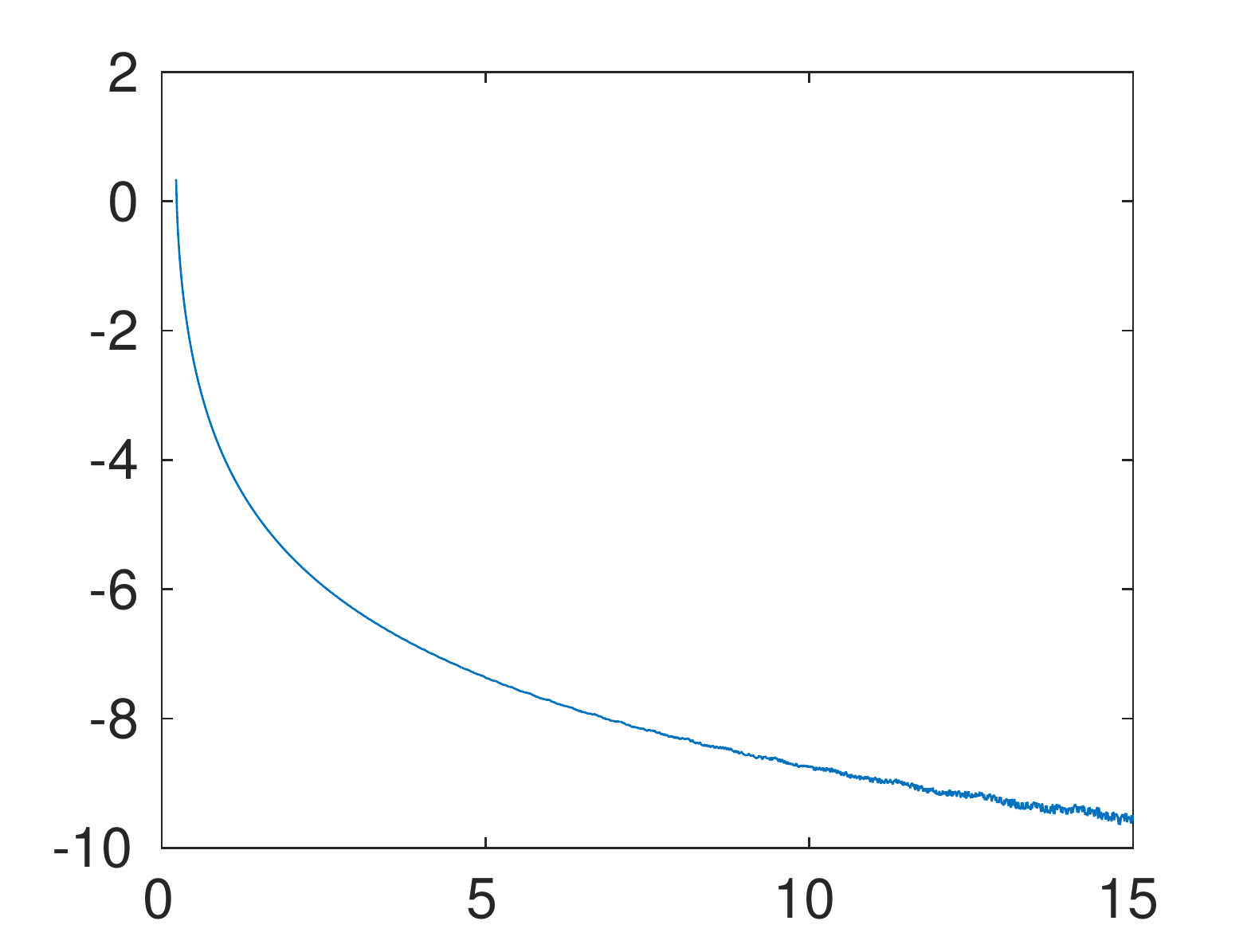}
\caption{{\it Left}: The curve $C \mapsto 4 C L_*(C)$ converges to the constant $\log\big(1 - F(1) / F(\theta) \big)$. {\it Right}: Visualization of the exponential rate of convergence: $C \mapsto \log \Big(4 C L_*(C) - \log\big(1-\f{F(1)}{F(\theta)}) \Big)$.}
\label{fig:6}
\end{figure}

\section{A brief recall on bistable reaction-diffusion in $\RR$.}
\label{sec:gen}

From now on we assume that
\beq\label{hypf}
f \text{ is Lipschitz},\,  f(0)=0 \text{ and }  f(1)=0. 
\eeq
We call $f$ {\bf monostable} if, in addition to \eqref{hypf},  $f>0$ on $(0,1)$.
We call $f$ {\bf bistable} if, in addition to \eqref{hypf}, there exists $\theta\in(0,1)$
such that $f(\theta)=0$,  $f<0$ on $(0,\theta)$
and $f>0$ on $(\theta,1)$. 
In all cases, we also assume that $f<0$ on $(-\infty,0)\cup (1,+\infty)$ (this is a technical assumption to facilitate some proofs,
$p$ being actually a frequency it will always remain between $0$ and $1$).

In the bistable case, we also assume $\int_0^1 f(x) dx > 0$ and define $\theta_c$ as the unique real number in $(0, 1)$ such that
$\int_0^{\theta_c} f(x) dx = 0$. (Obviously, $\theta_c > \theta$).
We define $F (x) = \int_0^x F(\xi) d\xi$, so that $F(\theta_c) = 0$.

We recall the following fact (see classical literature \cite{Fife1977} and \cite{Aronson1978} or \cite{chen1997} for a more recent proof)
\begin{proposition}[Bistable traveling wave]
 If $f$ is bistable, then there exists a unique $c= c_*(f)$, and a unique (up to translations) $p_*$ solution of
 \[
  -p_*'' - c p'_* = f(p_*) \text{ in } \RR, \quad p_*(-\infty) = 1, \, p_*(+\infty) = 0.
 \]
  In addition, $p_*$ is positive and decreasing. We call $c_*$ the {\it bistable wave speed}, and $p_*$
  the {\it bistable traveling wave}, because $u(t, x) = p_*(x - ct)$ is a solution to \eqref{eq:u} on $\RR$.
  \label{prop:bistbase}
\end{proposition}

\begin{definition}
 Let $\Omega \subset \RR^d$ be a regular, open set (bounded or not), $f : \RR \to \RR$ and $g : \p \Omega \to \RR$
 be two smooth functions.
 
 Let $\calL$ be an elliptic operator $\calL := \Delta + k(x) \nabla$, where $k$ is a smooth function $\Omega \to \RR$.
 
 A subsolution (resp. a supersolution) of the elliptic problem
 \beq
  - \calL u = f(u) \text{ in } \Omega, \quad u = g \text{ on } \p \Omega
  \label{eq:elliptic}
 \eeq
is $\usar \in \mathcal{C}^2 (\Omega) \cap \mathcal{C}^0 (\overline{\Omega})$ (resp. $\ubar \in \mathcal{C}^2 (\Omega) \cap \mathcal{C}^0 (\overline{\Omega})$) such that
\[
 - \calL \usar \leq f(\usar) \text{ in } \Omega, \quad \usar \leq g \text{ on } \p \Omega
\]
(respectively such that
\[
 - \calL \ubar \geq f(\ubar) \text{ in } \Omega, \quad \ubar \geq g \text{ on } \p \Omega.)
\]

  Similarly, a subsolution (resp. a supersolution) to the parabolic problem 
  \beq
\p_t u - \calL u = f(u) \text{ in } \Omega, \quad \forall t > 0, u(t, \cdot) = g(t, \cdot) \text{ on } \p \Omega, \quad u(0, \cdot) = u^0
(\cdot) \text{ in } \Omega.
  \label{eq:parabolic}
  \eeq
  is $\usar \in \mathcal{C}^1 \big(\RR_+; \mathcal{C}^2 (\Omega) \cap \mathcal{C}^0 (\overline{\Omega}) \big)$ such that
  \[
\p_t \usar - \calL \usar \leq f(\usar) \text{ in } \Omega, \quad \forall t > 0, \usar(t, \cdot) \leq g(t, \cdot) \text{ on } \p
\Omega, \quad \usar(0, \cdot) \leq u^0 (\cdot) \text{ in } \Omega.
  \]
  (respectively $\ubar \in \big(\RR_+; \mathcal{C}^2 (\Omega) \cap \mathcal{C}^0 (\overline{\Omega}) \big)$ such that
  \[
\p_t \ubar - \calL \ubar \geq f(\ubar) \text{ in } \Omega, \quad \forall t > 0, \ubar(t, \cdot) \geq g(t, \cdot) \text{ on } \p
\Omega, \quad \ubar(0, \cdot) \geq u^0 (\cdot) \text{ in } \Omega.)
  \]
  By definition, a solution is any function which is simultaneously a sub- and a super-solution.
\end{definition}

Sub- and supersolutions are used in the classical comparison principle:
\begin{proposition}[Sub- and super-solution method]
  Let $\usar$ be a subsolution (respectively $\ubar$ a supersolution) to \eqref{eq:elliptic}.
If $\usar < \ubar$ (which means $\usar(x) \leq \ubar(x)$ and $\ubar \not=\usar$) then there exist minimal and maximal solutions $u_*
\leq u^*$ such that $\usar \leq u_* \leq u^* \leq \ubar$.
 \label{prop:subsup}
  \end{proposition}
  
  \begin{proposition}[Parabolic comparison principle]
	For all $T > 0$ we introduce the ``parabolic boundary''
  \[
	\p_T \Omega :=  \Big( [0, T) \times \p \Omega \Big) \, \bigcup \, \Big( \{ 0 \} \times \Omega \Big).
  \]

If $\usar$ (resp. $\ubar$) is a sub-solution (resp. a super-solution) to \eqref{eq:parabolic}, and $u$ is a solution such that $u
\geq \ubar$ (resp. $u \leq \usar$) on $\p_T \Omega$ then the inequality holds on $\Omega \times [0, T]$.
 
 \label{prop:comparison}
\end{proposition}

In addition, the maximum (resp. the minimum) of two sub-solutions (resp. super-solutions) is again a sub-solution (resp. a
super-solution).
We also define the stability from below and above:
\begin{definition}
A solution $u$ to an elliptic problem is said to be stable from below (resp. above) if for all $\epsilon > 0$ small enough, there
exists a subsolution $\usar$ (resp. a supersolution $\ubar$) to the problem such that $u - \epsilon \leq \usar \leq u$. (resp. $u
\leq \ubar \leq u + \epsilon$).

It is said unstable from below (resp. above) if for all $\epsilon > 0$ small enough there exists a supersolution $\ubar$ (resp. a
subsolution $\usar$) to the problem such that $u - \epsilon \leq \ubar \leq u$ (resp. $u \leq \usar \leq u +\epsilon$).
\label{def:belabo}
\end{definition}

\section{Proofs for the infection-dependent population gradient model}
\label{sec:infdep}

We recall equation \eqref{eq:p}, in dimension $d=1$, for which we are going to prove Theorem \ref{thm:mainInf}
\[
 \p_t p - \p_{xx} p - 2 \f{h'(p)}{h(p)} \lvert \p_x p \rvert^2 = f(p).
\]

After giving an expression for $h$ in the case of {\it Wolbachia}, we prove that there exist traveling wave solutions to~\eqref{eq:p}, whose speed sign can be determined easily,
and eventually compared with traveling waves for \eqref{eq:u}. They can be initiated by ``$\alpha$-propagules'' (or
``$\alpha$-bubbles'') as in the case of \eqref{eq:u}, which was studied in \cite{BT} and \cite{SVZ}.
Due to the classical sharp-threshold phenomenon for bistable reaction-diffusion (see \cite{Zla} for the first proof with initial data as characteristic functions of intervals, \cite{Pol.Threshold} for extension to higher dimensions, \cite{DuMat.Convergence,Mat} and \cite{MurZho.Threshold} for extension to localized initial data in dimension $1$) solutions then have a simple asymptotic behavior.
The infection can either invade the whole space or extinct (or, for a ``lean'' set of initial data, converge to a ground state profile, and this is an unstable phenomenon).

  Hence when the population gradient is a function of the infection rate, there is no wave-blocking phenomenon.

\subsection{In the case of {\it Wolbachia}, $h$ is not monotone}
\label{subs:h}
Clearly, if $h$ is non-increasing, $h'\leq 0$, then the solution $p$ to \eqref{eq:p} is a sub-solution to
\eqref{eq:u}, assuming we complete them with the same initial data.
Hence $p\leq u$ for all time.

However, in the case of {\it Wolbachia}, the function $h$ (computed in the large population asymptotic developed in \cite{reduction}) is not monotone. It reads
$$
N=h(p)=1 - \epsilon\frac{d_u}{\sigma F_u} \frac{(\delta-1)p+1}
{s_h p^2-(s_f+s_h)p +1},
$$
hence
$$
h'(p)=\epsilon\frac{d_u}{\sigma F_u} \frac{(\delta-1)s_h p^2+2 s_h p -(\delta -1+s_f+s_h)}
{\big( s_h p^2-(s_f+s_h)p +1 \big)^2}.
$$
We can compute $h'(0)<0$, $h'(1)>0$, for $\delta s_h - \delta + 1 -s_f >0$ (this condition being necessary to ensure bistability in the limit equation, see details in \cite{reduction}).

We can show that $h'$ vanishes at a single point in $[0,1]$, where its sign changes.
This point is
$$\theta_0:=\ds \frac{1}{\delta-1}\left(-1+\sqrt{1+(\delta-1)(\frac{\delta-1+s_f}{s_h}+1)}\right)$$ 
for $\delta \neq 1,$
and if $\delta=1$, then $\theta_0=\ds \frac 12+\frac{s_f}{2s_h}$.

Hence if $p\leq \theta_0$ then $h'(p)\leq 0$.
As a consequence, for an initial datum $u^{\text{init}}=p^{\text{init}}$ 
such that  $\|p^{\text{init}}\|_\infty \leq \theta_0$, $p\leq u$ holds as long as $\|p\|_\infty \leq \theta_0$.
But no more can be said simply from \eqref{eq:u}.

\subsection{A change of variable to recover traveling waves}

\begin{proof}[Theorem \ref{thm:mainInf}]

First, we note that the function $H(x) = \int_0^x h^2 (\xi) d\xi$ is invertible on $[0, 1]$, since it is increasing ($h^2 > 0$).

Multiplying \eqref{eq:p} by $h^2(p)$ yields
$$
h^2(p)\pa_t p - \p_x (h^2(p) \p_x p ) = f(p) h^2(p).
$$
We set $y(x)=H(p(x))$ (equivalently, $p(x)=H^{-1}(y(x))$). 
Then
$$
\pa_t y - \p_{xx} y = f(H^{-1}(y)) h^2(H^{-1}(y)).
$$

And we are left with the following problem
\beq\label{eq:y}
\pa_t y -\p_{xx} y = g(y), \qquad g(y)=f(H^{-1}(y)) h^2(H^{-1}(y)).
\eeq

Since $f$ is defined on $[0,1]$, $g$ is also defined on $[H(0),H(1)]=[0,1]$. Because of \eqref{hypf},
$$
g(0)=g(H(0))=0, \quad g(1)=g(H(1)) = 0, \quad g \mbox{ has the same sign as } f\circ H^{-1}.
$$

Hence if $f$ is monostable then $g$ is monostable. If $f$ is bistable with
$f(\theta)=0$ for some $\theta\in (0,1)$, then $g$ is also bistable with
$g(H(\theta))=0$, and $H(\theta)\in (H(0),H(1)) = (0, 1)$.

We compute
$$
g'(y)=f'\big(H^{-1}(y)\big)+ 2 f(H^{-1}(y))
\frac{h'(H^{-1}(y))}{h(H^{-1}(y))}.
$$
In particular, $g'(0)= f'\big(0\big)$.

Obviously, if there exists a traveling wave for \eqref{eq:y}, $y(t,x)=\widetilde{y}(x-c t)$,
connecting $1$ to $0$, then
$p(t,x):=H^{-1}\big(H(0)+(H(1)-H(0)) \widetilde{y}(x-c t)\big)$ is a traveling wave for
 \eqref{eq:p}, connecting $1$ to $0$.

Then we can compare the wave speeds for \eqref{eq:y} and for \eqref{eq:u}.
\begin{enumerate}
\item If $f$ is monostable, then there exists a minimal traveling speed $c^*$.
such that for all $c\geq c^*$, there exists a unique, decreasing, traveling wave
$0\leq y \leq 1$ for \eqref{eq:y}, connecting $1$ to $0$.
Moreover, if KPP condition $g(x)\leq g'(0) x$ holds on $[0,1]$, then
$c^* = 2\sqrt{g'(0)}=2\sqrt{f'(0)}$.

We notice that the KPP condition $g(x)\leq g'(0) x$ for all $x \in (0, 1)$ holds if and only if
$f(z) h^2(z) \leq H(z) f'(0)$, by setting $z=H^{-1}(x)$.
Hence if $f$ itself satisfies the KPP condition, {\it i.e. }satisfies $f(z)\leq f'(0) z$, it suffices to check
$h^2(z) \leq \frac{H(z)-H(0)}{z}$, $\forall\,z\in(0,1)$. 
This condition is equivalent to concavity of $H$ on $(0,1)$, {\it  i.e. }$h'\leq 0$
on $(0,1)$.

\item If $f$ is bistable, then there exists a unique traveling wave $(c_*,v)$
for \eqref{eq:y}, decreasing, connecting $1$ to $0$ and
$
c_* <0 \mbox{ if } G(1)<0, \quad c^* =0 \mbox{ if } G(1)=0, \quad
c_* >0 \mbox{ if } G(1)>0,
$
where $G(1)=\int_0^1 g(v)\,dv$ (see \cite{Perthame}).
Using the definition of $g$ in \eqref{eq:y} we get
$$
G(1)=\int_{0}^{1} g(y) dy = \int_0^1 f(x) h^4(x) dx.
$$
\end{enumerate}

\end{proof}

\begin{remark}
If $h \equiv 1$ then $H=Id$ and we recover $f=g=\widetilde{g}$.
\end{remark}

\begin{remark}
In the monostable case we find $c^*= 2\sqrt{f'(0)}$, so the minimal speed for \eqref{eq:p} and for \eqref{eq:u} are the same.
\end{remark}

If $f$ is bistable and $G(1)>0$, the sharp threshold property (see \cite{Mat}) applies to equation \eqref{eq:y}, hence to equation \eqref{eq:p}.

\subsection{Critical propagule size}

To identify the initial data that induce invasion, we can compute ``propagules'' (also called ``bubbles''), that is compactly supported subsolutions to the parabolic problem \eqref{eq:p}. This was stated in Proposition \ref{prop:propagules}, that we are going to prove below.

The concept of critical propagule size, that is the minimal ``size'' of an initial data to ensure invasion, was studied in \cite{BT}. 
We reproduce here for equation \eqref{eq:y} the computations that can be found in \cite{BT} and \cite{SVZ}, and deduce an expression of the critical propagule for equation \eqref{eq:p}.

\begin{proof}[Proposition \ref{prop:propagules}]
We introduce the following Cauchy system associated with \eqref{eq:p}
\beq
\bepa
p''+ 2 \displaystyle\frac{h'(p)}{h(p)} (p')^2 + f(p) = 0, \qquad \mbox{ on } [0,+\infty) 
\\[10pt]
 p(0)=\alpha, \qquad p'(0)=0
\eepa
\label{eq:propagule1}
\eeq
Multiplying equation \eqref{eq:propagule1} by  $h(p)^2$  yields
$
\big(h(p)^2 p'\big)'= -f(p)h(p)^2.
$
Then, multiplying by $h(p)^2 p'$ and integrating over $[0,x)$ yields
$$
\frac{1}{2} \Big(\big(h(p)^2 p')^2 - \big(h(\alpha)^2 p'(0)\big)^2 \Big)=
-\mathcal{F}(p) + \mathcal{F}(p(0)),
$$
where $\mathcal{F}$ is an antiderivative of $p\mapsto f(p)h(p)^4$.

We are looking for a decreasing solution $p$ on $[0, +\infty)$.
Since $p'(0)=0$ we get
$$
p' = - \frac{\sqrt{2(\mathcal{F}(\alpha) - \mathcal{F}(p))}}{h(p)^2}.
$$
Note that since $h(p)^4 > 0$, $\mathcal{F}'$ has the same sign as $f$.
If $h$ is constant, we recover the case of equation \eqref{eq:p} without correction term.

We make a change of variable and check that $v_{\alpha} := \max(p, 0)$ has support equal to $[0,L_\alpha]$ where
\beq
L_\alpha := \int_0^\alpha \frac{h(p)^2}{\sqrt{2(\mathcal{F}(\alpha)-\mathcal{F}(p))}}\,dp.
\label{eq:Lalpha}
\eeq
As for the ``classical case'' (without $h$) treated in \cite{SVZ}, convergence of this integral is straightforward (recalling $\alpha > \theta$). Thus $L_{\alpha} < \infty$.

Hence we constructed a family $(v_{\alpha})_{\theta_c < \alpha < 1}$ of compactly supported sub-solutions,
where $0 \leq v_{\alpha} \leq \alpha$.
\end{proof}

\section{Proofs for the heterogeneous case: blocking waves and barrier sets}
\label{sec:het}

This section is devoted to the proof of the main results concerning existence of blocking fronts, {\it i.e.} Theorem~\ref{thm:mainHet} and Proposition \ref{prop:barriers}.
This proof is divided in several steps. In Subsection \ref{subs:prel} we prove Proposition \ref{prop:passblock} and first point of Proposition \ref{prop:barriers}.
In Subsection \ref{subs:shoot} we reformulate the existence problem as a double shooting problem and establish the first properties.
In Subsection \ref{phaseplane} we introduce a phase-plane method.
This allows us to state useful properties on the barrier set.
Then, Theorem \ref{thm:mainHet} and Proposition \ref{prop:Lstarasympt} are proved in Subsection~\ref{subs:advanced},
whereas Proposition \ref{prop:barriers} is proved in Subsection \ref{subs:gathering}.
We conclude in Subsection \ref{subs:generalizing} by proving Corollary \ref{cor:nonconstantgradient}.

\subsection{Preliminaries}
\label{subs:prel}

 The first fact we prove about the barriers (see Definition \ref{def:barriers}) is that they are decreasing. This is the first point of Proposition \ref{prop:barriers}.
 \begin{lemma}
  If $(C, L) \in \calB (f)$ and $p$ is a $(C,L)$-barrier, then $p$ is decreasing.
  \label{lem:decreasingbarrier}
 \end{lemma}
 \begin{proof}
 For any  $x \in (- \infty, -L]$, we have
 \[
  \f{1}{2} p' (x)^2 + F( p(x) ) = F(1).
 \]
 Hence $p' = 0$ if and only if $p(x) = 1$, but the maximum principle forbids it ($1$ is a super-solution so $p$ cannot touch it).
 
 Similarly, $p'$ does not change its sign on $[L, +\infty)$, except possibly if $p = \theta_c$ or $p = 0$.
 $p=0$ is impossible by the same argument as before. Assume $p(L) < \theta_c$. Then:
  \[
    \f{1}{2} p'(L)^2 + F(p(L)) = F(0) = 0.
  \]
  In addition we claim $p'(L) < 0$.  
  To prove this last fact we introduce
  \[
   x_m := \inf \{ x > -L, p'(x) = 0 \}.
  \]
  By contradiction, we assume $x_m < L$. There are two possibilities.
  
  Either $p(x_m) < \theta_c$. In this case, $\f{1}{2} p'(x_m)^2 + F (p(x_m)) < 0$.
  Since $\psi : x \mapsto \f{1}{2} p'(x)^2 + F(p(x))$ is decreasing and is equal to $0$ at $x= L$,
  this contradicts $x_m < L$. (Indeed, for all $x \in (-L, L)$,  $\psi(x) = F(1) - C_N \int_{-L}^x p'(x')^2 dx'$.)
   Or $p(x_m) \geq \theta_c$. If $1 > p(x_m) \geq \theta_c$ then $-p''(x_m) = -p''(x_m) - C p'(x_m) = f(p(x_m)) > 0$, hence
  $p$ reaches a local maximum at $x_m$, which is absurd because this contradicts the definition of $x_m$.
  Hence $p' < 0$ on $[-L, L]$.  
  
  Because $0 \leq p \leq 1$ and because of its limits at $\pm \infty$, $p$ is
 necessarily decreasing on $(-\infty, -L] \cup [L, +\infty)$.
\end{proof}

Existence of a barrier means that the (logarithmic) gradient of total population is enough to stop the bistable propagation.
On the contrary, when there is no barrier, then bistable propagation takes place. This is the object of Proposition \ref{prop:passblock}, which we prove below.

\begin{proof}[Proposition \ref{prop:passblock}]
The first point comes directly from the comparison principle (Proposition \ref{prop:comparison}), since $p_B$ is a stationary
solution, hence a super-solution to \eqref{eq:onR}. It is easily checked that $p_{B}<1$ by considering a maximum of this function.

First, assume $(C, L) \in \calB(f)$ and $p^0 > p_B$ for the maximal barrier $p_B$.
By hypothesis, it is unstable from above, hence there exists a sub-solution $\phi$ to \eqref{eq:TW} between $p_B$ and $p^0$.
Hence by the comparison principle $p(t, \cdot)$ is bounded from below by $p_{\phi} (t, \cdot)$, for all $t \geq 0$, where $p_{\phi}$
is the solution to \eqref{eq:onR} with initial datum $\phi$.
Since $p_{\phi}$ is increasing in $t$ (because initial datum is a subsolution), it converges to some $p_{\phi}^*$ as $t \to \infty$.
However, $p_{\phi}^*$ is a solution to \eqref{eq:TW} with the last hypotheses on $p(\pm\infty)$ relaxed. Because $p_B$ is a maximal
barrier (there is no element above it) and $p_{\phi}^* > p_B$,
this implies that $p_{\phi}^* (+\infty)$ is a zero of $f$ which is not $0$, hence it must be either $\theta$ or $1$. Since
$-p_{\phi}'' = f(p_{\phi})$ on $[L, +\infty)$, $p_{\phi}^*$ has to go below $\theta$. Otherwise
it is decreasing (by Lemma \ref{lem:decreasingbarrier}) and concave, hence cannot converge to a finite value.

Finally, if $(C, L) \not\in \calB(f)$, because $p_0 > p_B$ or $\lim_{-\infty} p_0 = 1$, we can always pick a sub-solution $\phi$
which is below $p^0$. For example, a translated $\alpha$-bubble (from Proposition \ref{prop:propagules} in the case $h = 0$) $v_{\alpha}
(\cdot - \tau)$ for some $\tau > 0$ large enough. The solution to \eqref{eq:onR} with initial datum $\phi$, say $p_{\phi} (t, \cdot)$
is increasing in $t$, and by the comparison principle it is below $p$ for all $t$.
Because it is increasing, its limit as $t \to \infty$ is well-defined and it is a solution to \eqref{eq:TW} without the final
conditions (on $p(\pm \infty)$). Since \eqref{eq:TW} has no solution, this implies that $p_{\phi} (t, \cdot) \to 1$. Hence $p \to 1$.
\end{proof}

To simplify notably the study of the barrier set $\calB(f)$, we first obtain a simple positivity property by using the comparison
principle (Proposition \ref{prop:comparison})
and the super- and sub-solutions method.

\begin{proposition}
 For all $B_1 \in \calB (f)$ and $B_2 \in [0, +\infty)^2$, $B_1 + B_2 \in  \calB (f)$.
 \label{prop:positive}
\end{proposition}
\begin{proof}
Let $B_1 = (C_1, L_1)$, $p_1$ be a solution to \eqref{eq:TW} where $C= C_1$ and $L = L_1$.
Let $B_2 = (C_2, L_2)$.
Then, $p_1$ is \textit{decreasing} (by Lemma \ref{lem:decreasingbarrier}), hence
\begin{align*}
 - p_1 '' - (C_1+ C_2) p_1 ' &\geq p_1 '' - C_1 p_1 ' = f(p_1) \text{ on } [-L_1, L_1],\\
 - p_1 '' - (C_1 + C_2) p_1 ' &\geq p_1 '' = f(p_1) \text{ on } [-(L_1 + L_2), -L_1] \bigcup [L_1, L_1 + L_2],\\
 - p_1 '' &= f(p_1) \text{ on } \RR \backslash [-(L_1 + L_2) , L_1 + L_2].
\end{align*}
In other words, $p_1$ is a supersolution of \eqref{eq:TW} for $C = C_1 + C_2$, $L = L_1 + L_2$.

On the other hand, the $\alpha$-bubbles from Proposition \ref{prop:propagules} give us subsolutions, and we can select any of them.
Upon moving it far enough towards $-\infty$, it will be below $p_1$.
We simply need to consider $v_{\alpha} (\cdot - \tau)$ for $\tau > 0$ large enough, which will be the required subsolution.

This implies that we can construct a solution $p$ to \eqref{eq:TW} for $C = C_1 + C_2$ and $L = L_1 + L_2$,
lying between the $\alpha$-bubble and $p_1$, by Proposition \ref{prop:subsup}.
As $p_1$ is decreasing, one could check that $p$ is decreasing as well, and thus it admits limits at $\pm \infty$.
Then one could check that $p(+\infty) = 0$ and $p(-\infty)=1$, whence $p$ is a barrier.
Hence $B_1 + B_2 \in \calB(f)$.
\end{proof}

\subsection{A double shooting-argument.}
\label{subs:shoot}
To get a better description of $\calB(f)$, we introduce a double shooting-argument.
We separate the study of equation~\eqref{eq:TW} on $[-L, L]$ by introducing 
\[
 \beta = p(-L), \quad \alpha = p(L).
\]
We are left with a slightly differently rephrased problem: given $0 < \alpha < \beta < 1$, we are looking for
$C, L > 0$ such that
\beq
\bepa
 - p'' - C p' = f(p),
 \\[10pt]
 p(-L) = \beta, \quad p(L) = \alpha,
 \\[10pt]
 \f{1}{2} p'(-L)^2 + F(\beta) = F(1), \quad \f{1}{2} p'(L)^2 + F(\alpha) = 0.
\eepa
\label{eq:alphabeta}
\eeq

The two equations \eqref{eq:TW} and \eqref{eq:alphabeta} are obviously directly related.
\begin{proposition}
 Let $C, L > 0$. If $(C, L) \in \calB(f)$, then there exists $(\alpha, \beta)$ such that  \eqref{eq:alphabeta} has a solution.
 Conversely, if there are $\alpha, \beta$ and $C, L$ such that \eqref{eq:alphabeta} has 
 a solution, then its solutions are also solutions to \eqref{eq:TW}.
\end{proposition}
The proof is a straightforward computation.
A first property of \eqref{eq:alphabeta} can easily be proven:
\begin{proposition}
For any $0 < \alpha < \beta < 1$ with $\alpha < \theta_c$, there exists a unique $C = \gamma(\alpha, \beta)$ such that the system
\eqref{eq:alphabeta}
has a solution, associated with a unique $L = \lambda(\alpha, \beta)$.
\label{prop:exists}
\end{proposition}
\begin{proof}
Here we employ a shooting argument.
Let $p_{\alpha}$ be the unique (by Cauchy-Lipschitz theorem), decreasing (by similar arguments as in Lemma \ref{lem:decreasingbarrier}) solution to
\beq
\bepa
-p_{\alpha}'' - C p_{\alpha}' = f(p_{\alpha}),
\\[10pt]
p_{\alpha}(L) = \alpha, \quad \f{1}{2} p_{\alpha}' (L)^2 + F(\alpha) = 0.
\eepa
\label{eq:palpha}
\eeq
Because $p_{\alpha}$ is decreasing, we can introduce $X_{\alpha} : [p_{\alpha}(L), p_{\alpha}(-L)] \to [-L, L]$ such that
$p_{\alpha}(X_{\alpha}(p)) = p$.
Using the method of \cite{BerNicSch} we also introduce $w_{\alpha} (p) := \f{1}{2} p_{\alpha}' (X_{\alpha}^{-1} (p) )^2 + F(p)$.
Then:
\beq
\bepa
w_{\alpha}' (p) = C \sqrt{2 \big( w_{\alpha}(p) - F(p) \big)},
\\[10pt]
w_{\alpha}(\alpha) = 0.
\eepa
\label{eq:w}
\eeq

The solution of this problem exists as long as $w_{\alpha}(p)\geq F(p)$. 
For $\alpha <\theta_c$, since $F(p)<0$ for $p\in(0,\theta_c)$, we deduce that 
the solution exists at least on $(\alpha,\theta_c)$.
Let us denote $p_0\leq 1$ such that $(\alpha,p_0)$ is the maximum interval in $(\alpha,1)$
of existence of a solution to \eqref{eq:w}. We have $p_0\geq \theta_c$.

Then, let $\beta > \alpha$. We are going to show that we can choose $C$ such that $w_{\alpha} (\beta) = F(1)$.
We first notice that on $(\alpha,\theta_c)$, we have $F(p)<0$ thus 
$w'_{\alpha}(p)>C\sqrt{2w_{\alpha}(p)}$. It implies that $w_{\alpha}(p)> \frac{1}{2} C^2 (p-\alpha)^2$ 
on $(\alpha,\theta_c)$.
Thus if $C$ is large enough, surely we will have $w_{\alpha}(\beta)>F(1)$.

Conversely, we have $w'_{\alpha}(p)\leq C \sqrt{2(w_{\alpha}(p)-F(\theta))}$, since $F(\theta)=\min_{[0,1]} F$.
Integrating on $(\alpha,p)$, we deduce 
$w_{\alpha}(p)\leq F(\theta)+ \big(\frac{1}{\sqrt{2}} C (p-\alpha)+\sqrt{-F(\theta)}\big)^2$.
Thus we may choose $C$ small enough such that $w_{\alpha}(\beta)<F(1)$.
Finally, by deriving \eqref{eq:w} with respect to $C$, we deduce that 
the solution $w$ is increasing with respect to $C$.

Hence for each $\beta$ there exists a unique $C = \gamma(\alpha, \beta)$ such that $w_{\alpha} (\beta) = F(1)$. We rename this solution as $w_{\alpha, \beta}$, so that
\beq
\bepa
w'_{\alpha, \beta} (p) = \gamma (\alpha, \beta) \sqrt{2 \big(w_{\alpha, \beta} (p) - F(p) \big) },
\\[10pt]
w_{\alpha, \beta} (\alpha) = 0, \, w_{\alpha, \beta} (\beta) = F(1).
\eepa
\label{eq:w2}
\eeq

To retrieve the value of $L$, such that $w_{\alpha, \beta}$ comes from a $p_{\alpha}$ 
solution of \eqref{eq:palpha} with $p_{\alpha} (-L) = \beta$, 
$\f{1}{2} \big( p'_{\alpha} (-L) \big)^2 + F(p_{\alpha}(-L)) = F(1)$, 
we simply have to remark that $L = \f{1}{2} \int_{\beta}^{\alpha} \big( X_{\alpha}^{-1} \big)' (p) dp$.
To compute it from $w_{\alpha, \beta}$ we notice that $(X_{\alpha}^{-1}) ' (p) = 1 / p_{\alpha}' \big(X_{\alpha}^{-1} (p) \big)$.
Hence we define
\beq
 \lambda (\alpha, \beta) := \f{1}{2} \int_{\alpha}^{\beta} \f{1}{\sqrt{2 \big( w_{\alpha, \beta}(p) - F(p) \big) }}  dp.
 \label{eq:lambda}
\eeq
(Indeed, recall that $p' < 0$ on $(-L, L)$). Then $L = \lambda(\alpha, \beta)$ is uniquely defined.

\end{proof}

\begin{lemma}
Functions $\gamma$ and $\lambda$ defined in Proposition \ref{prop:exists} are continuous on $\{ (\alpha, \beta), \, 0 < \alpha <
\theta_c, \text{ and } \alpha < \beta < 1 \}$.
\label{lem:continuitygl}
\end{lemma}
\begin{proof}
We transform problem \eqref{eq:w2} into a ordinary differential equation $w' (p) = \gamma J(w(p), p)$, with either $w(\alpha) = 0$ or
$w(\beta) = F(1)$, and $\gamma > 0$.

On the prescribed set for $\alpha, \beta$, the function $J$ is uniformly Lipschitz along any forward trajectory.
This implies the continuity of $w$ with respect to $\gamma$, and finally the continuity of $\gamma$ with respect to $\beta$ (in the
case when we impose $w(\alpha) = 0$), and with respect to $\alpha$ (when we impose $w(\beta) = F(1)$).

This implies the continuity of $\lambda$.
\end{proof}

\begin{proposition}
 Let $L > 0$. If $(C, L) \in \calB(f)$ then $C > c_* (f)$.
 \label{prop:cstar}
 \end{proposition}
\begin{proof}
This comes from the fact that there exists $w_{0, 1}$ such that
\beq
\bepa
w'_{0, 1} = c_*(f) \sqrt{2 ( w_{0, 1} - F)},
\\[10pt]
w_{0, 1} (0) = 0, \, w_{0, 1} (1) = F(1).
\eepa
\eeq
And the associated $\lambda(0, 1)$ is equal to $+\infty$.
By comparison of solutions to \eqref{eq:w2}, no $(\alpha, \beta) \not=(0, 1)$ could give a $w_{\alpha, \beta}$ associated with $C \leq c_*(f)$.
\end{proof}

\subsection{A graphical digression on phase plane analysis.}
\label{phaseplane}
Equation \eqref{eq:alphabeta} can be easily interpreted in the phase plane $(p, p')$. 
For this interpretation, we follow the presentation of \cite{lewiskeener}.
Let $X=p$, $Y = p'$. The equation rewrites into the system
\beq
\bepa
X ' = Y, \quad X(0) = X_0,
\\[10pt]
Y ' = - C Y - f(X), \quad Y(0) = Y_0.
\eepa
\label{sys:XY}
\eeq
The energy $E : \RR^2 \to \RR$ may be defined as
\beq
E (X, Y) := \f{1}{2} Y^2 + F(X).
\label{eq:energy}
\eeq

Two interesting curves appear:
\begin{align}
E^{-1} \big( F(1) \big) \supset \Gamma_{B} &:= \Big\{ (x, y) \in [0, 1] \times (-\infty, 0], \, y = - \sqrt{2\big(F(1) - F(x) \big)} \Big\},
\\
 E^{-1} \big( 0 \big) \supset \Gamma_{A} &:= \Big\{ (x, y) \in [0, \theta_c] \times (-\infty, 0], \, y = - \sqrt{-2 F(x)} \Big\}.
\end{align}

A $(C, L)$-barrier can be seen there as a trajectory of system \eqref{sys:XY} with $\big(X(-L), Y(-L) \big) \in \Gamma_{B}$
such that $\big(X(L), Y(L) \big) \in \Gamma_A$.
Therefore, we are left studying the image of $\Gamma_B$ by the flow of \eqref{sys:XY},
which we denote by $\phi^C_t : \RR^2 \times \RR^2$, at time $t$.

\begin{lemma}
 The energy decreases along trajectories:
 \[
  \f{d}{dt} E \big( X(t), Y(t) \big) = - C Y(t)^2.
 \]
 At the three equilibrium points of the system it is equal to:
 \[
  E(0, 0) = 0, \, E(\theta, 0) = F(\theta) < 0, \, E(1, 0) = F(1) > 0.
 \]
 It is therefore minimal at $(\theta, 0)$.
 \label{lem:decrenergy}
\end{lemma}
This is a straightforward computation.

Let $\chi \in [\theta_c, 1]$. We define the level set of $E$
\[
\Gamma_{\chi} := E^{-1} \big( F(\chi) \big) = \Big\{ (x, y) \in [0, \chi] \times (-\infty, 0], y = - \sqrt{2 \big( F(\chi) - F(x) \big)}
\Big\} .
\]
Note that $\Gamma_1 = \Gamma_A$ and $\Gamma_{\theta_c} = \Gamma_B$, by definition.

For $\chi \in [\theta_c, 1]$ and $P \in \Gamma_{\chi}$, let $\nu_{\chi} (P)$ be the inward normal vector (``inward'' meaning pointing
towards $y=0$).
Then we claim
\begin{lemma}
For all $\chi \in [\theta_c, 1]$, $P \in \Gamma_{\chi}$, $C > 0$, the flow of \eqref{sys:XY} is inward: $\f{d}{dt}\phi^C_0 (P) \cdot
\nu_A (P) > 0$.
 \label{lem:inward}
\end{lemma}
\begin{proof}
First, system \eqref{sys:XY} may be rewritten $\dot{u} = G(u), u(0) = u_0$, where $u = (X, Y)$ and $u_0 = (X_0, Y_0)$.
Then, $\f{d}{dt} \phi^C_0 (u_0) = G(u_0)$, obviously (and similarly, $\f{d}{dt} \phi^C_t (u_0) = G( u(t))$).

Now, recall that $\Gamma_{\chi} = \{ (\alpha, -\sqrt{2\big( F(\chi) -  F(\alpha) \big)} \text{ where } 0 \leq \alpha \leq \chi \}$.
Hence if $P = (\alpha, -\sqrt{2 (F(\chi) - F(\alpha) )})$,
\[
 \nu_{\chi} (P) = \begin{pmatrix}
              - \displaystyle\f{f(\alpha)}{\sqrt{2(F(\chi) -  F(\alpha))}} \\
              1
             \end{pmatrix}
             \]
and
\[
  \f{d}{dt} \phi^C_0 (P) = G(p) =  \begin{pmatrix}
                    - \sqrt{2 (F(\chi) - F(\alpha))} \\
                    C \sqrt{2 (F(\chi) - F(\alpha))} - f(\alpha)
                   \end{pmatrix}.
\]
Hence
\[
 D \phi_0^C (P) \cdot \nu_{\chi} (P) = C \sqrt{2 (F(\chi) - F(\alpha))} > 0.
\]

\end{proof}

The following crucial property will make us able to show that barriers are ordered.
Its graphical interpretation is shown on Figure \ref{fig:sketch}.

\begin{lemma}
Let $p_1, p_2 \in (0, 1)$ with $p_1 < p_2$. We denote $(X_1, Y_1)$ (resp. $(X_2, Y_2)$) the unique solution of \eqref{sys:XY} with
$X_1 (0) = p_1$ (resp. $X_2 (0) = p_2$) and $Y_1 (0) = - \sqrt{2 \big( F(1) - F(p_1) \big)}$ (resp. $Y_2 (0) = - \sqrt{2 \big( F(1) -
F(p_2) \big)}$.

Let $t_M > 0$ be such that for all $t < t_M$, $Y_1, Y_2 < 0$, $X_1, X_2 > 0$. Then

\beq
\forall t < t_M, \quad X_1 (t) < X_2 (t).
\eeq
\label{keyineq}
\end{lemma}

\begin{figure}[h]
\centering
\includegraphics[width=.8\textwidth]{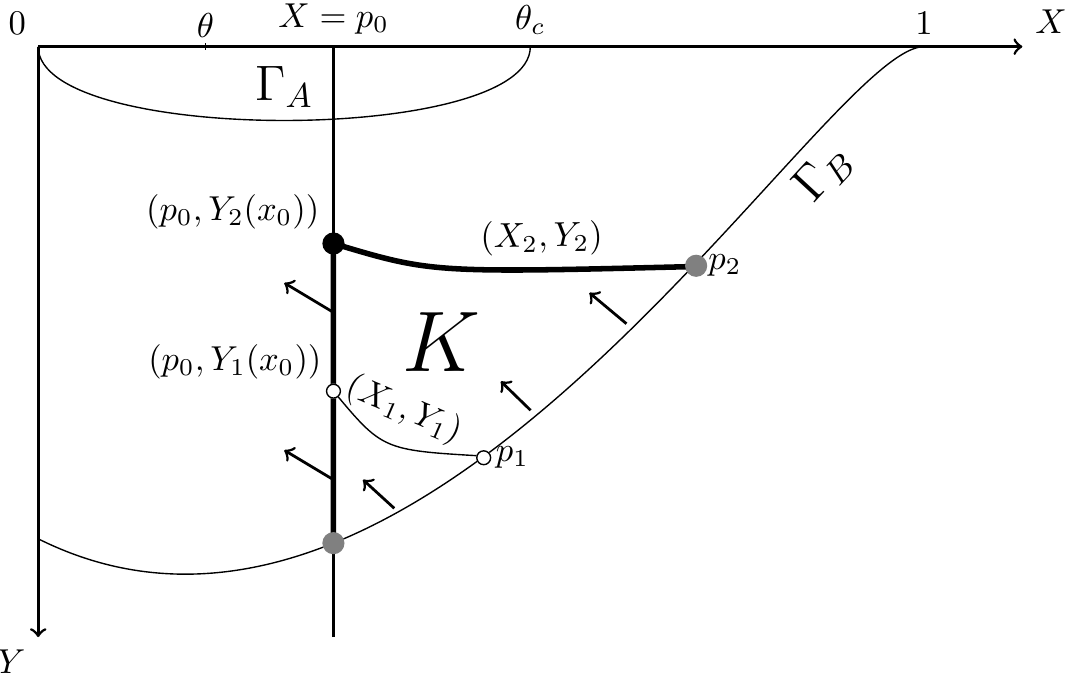}
\caption{Sketch of the phase-plane argument in the proof of Lemma \ref{keyineq}. Because the trajectories satisfy $\dot{X}=Y$, this picture is impossible. On the other hand, $Y_1(x_0) > Y_2 (x_0)$ would imply that the two trajectories cross each other, which is impossible as well. Whence the claim.}
\label{fig:sketch}
\end{figure}

\begin{proof}
To prove this we introduce
\[
	t_0 := \inf \{ t > 0, \, X_1 (t) = X_2 (t) \}.
\]
If $t_0 = + \infty$, we are done. If $t_0 < +\infty$, we first note that if $t < t_0$ then $X_1 (t) < X_2 (t)$, by definition of $t_0$
and continuity of $X_1, X_2$. As a consequence, $\f{d}{dt} (X_2 - X_1) (t_0) \leq 0$, and $Y_1 (t_0) \geq Y_2 (t_0)$.

We show that phase-plane reasoning imposes
\[
	Y_1 (t_0) \leq Y_2 (t_0).
\]
To prove this fact, we first observe that \eqref{sys:XY} has its flow from the right to the left along any vertical line ($X =$
constant), in the quadrant $X> 0, Y<0$ (because $\dot{X} = Y$).

Moreover, $Y_2 (t_0) > - \sqrt{2 \big( F(1) - F(t_0) \big)}$, because $E(X_2 (t_0), Y_2 (t_0) ) < F(1) = E(X_2 (0), Y_2(0) )$, by Lemma \ref{lem:decrenergy} ($E$ was defined in \eqref{eq:energy}).

Hence the trajectory of $(X_1, Y_1)$ enters at $x=0^+$ the compact set $K$ defined by the vertical line $X = X_1 (t_0)$, the
trajectory of $(X_2, Y_2)$ and $\Gamma_B$ (that is, the level set $F(1)$ of $E$). Indeed, $(X_1 (0), Y_1 (0))$ is on the part of
$\Gamma_B$ which defines the border of $K$, and the flow of \eqref{sys:XY} is inward at this point (by Lemma \ref{lem:inward}).

Moreover the trajectory of $(X_1, Y_1)$ cannot exit $K$ but on the line $X = X_1 (x_0) =: p_0$: its energy decreases and it cannot
cross the trajectory of $(X_2, Y_2)$. More precisely, it exits $K$ on the segment
\[
\big[ \big( p_0, \, - \sqrt{2 ( F(1) - F(t_0) )}\big), \quad \big( p_0, \, Y_2 (t_0) \big) \big] \subset \{ X = p_0 \}.
\]
As a consequence, $Y_1 (t_0) \leq Y_2 (t_0)$. 
 
 Hence $Y_1 (t_0) = Y_2 (t_0)$, which contradicts the uniqueness of the solutions of \eqref{sys:XY} (since $X_1 (t_0) = X_2 (t_0)$).
Finally, $t_0 = +\infty$ and Lemma \ref{keyineq} is proved.
\end{proof}

\subsection{Back to the double-shooting.}
Thanks to the double-shooting argument, determining $\calB(f)$ amounts to computing the image of $\{ 0 < \alpha < \beta < 1, \,
\alpha < \theta_c\}$ by $(\gamma, \lambda)$.

These functions $\gamma, \lambda$ have nice monotonicity properties.
\begin{proposition}
Let $\gamma$ and $\lambda$ be defined as in Proposition \ref{prop:exists} on the set $\{ (\alpha, \beta) \in (0, 1)^2, \quad 0 \leq \alpha \leq \theta_c, \, \beta > \alpha \}$.
 $\gamma (\alpha, \beta)$ is increasing in $\alpha$, decreasing in $\beta$.
 $\lambda (\alpha, \beta)$ is increasing in $\beta$.
 \label{prop:glmonotonicity}
\end{proposition}
\begin{proof}
Take $0 < \alpha < \beta$ with $\alpha < \theta_c$, $C = \gamma( \alpha, \beta)$ and $w$ be the solution of \eqref{eq:w2} associated
with $C$ and $\beta$.
Similarly, take $\tilde{\beta} > \beta$ and let $\tilde{C}:= \gamma(\alpha,\tilde{\beta})$ and 
$\tilde{w}$ the solution of \eqref{eq:w2} associated with $\tilde{C}$ and $\tilde{\beta}$ ({\it i.e.} $\tilde{w} (\tilde{\beta}) =
F(1)$).
 Assume by contradiction that $\tilde{C}\geq C$. 
 Then $\tilde{w}$ is a supersolution of the equation satisfied by $w$, 
 with initial datum $\tilde{w}(\alpha)=0$. 
 Hence $\tilde{w} \geq w$ on $[\alpha,\beta]$ and 
 $\tilde{w}(\beta) \geq F(1) = \tilde{w} (\tilde{\beta})$. 
 This is a contradiction since $\tilde{w}$ is increasing. 
 
 Hence, $\tilde{C}< C$ and thus, as $w(\alpha)=\tilde{w}(\alpha)=0$, one gets $\tilde{w}<w$ on $(\alpha, \beta)$. 
 We can therefore compute 
$$\lambda (\alpha,\tilde{\beta}) = \int_{\alpha}^{\tilde{\beta}} \f{dx}{\sqrt{2 \big(\tilde{w}(x) - F(x) \big)}} >
\int_{\alpha}^{\beta} \f{dx}{\sqrt{2 \big(w(x) - F(x) \big)}}=\lambda (\alpha, \beta),$$
 proving the monotonicity of $\lambda$ as a function of $\beta$. 

 The monotonicity of $\gamma$ with respect to $\alpha$ is proved similarly. 
\end{proof}

\begin{proposition}
  Functions $\gamma, \lambda$ satisfy:
 $\gamma(\alpha, \beta) \to +\infty$ as $\beta \searrow \alpha$.
 
 $\lambda (\alpha, \beta) \to +\infty$ as $\beta \to 1$, $\lambda (\alpha, \beta) \to + \infty$ as $\alpha \to 0$.
 $\lambda(\alpha, \beta) \to 0$ as $\beta - \alpha \to 0$.
 \label{prop:gllimits}
\end{proposition}
\begin{proof}
We have already proved in Proposition \ref{prop:exists} that 
\[
w(p)\leq F(\theta)+ \big(\frac{1}{\sqrt{2}} \gamma(\alpha,\beta) (p-\alpha)+\sqrt{-F(\theta)}\big)^2 .
\]
Hence, taking $p=\beta$, one has 
$F(1)- F(\theta) \leq \big(\frac{1}{\sqrt{2}} \gamma(\alpha,\beta) (\beta-\alpha)+\sqrt{-F(\theta)}\big)^2$.
If $\gamma(\alpha,\beta)$ does not diverge to $+\infty$ when $\beta \searrow \alpha$, this function would be bounded since it is
monotonic, and thus, passing to the limit in the inequality:
$F(1)- F(\theta) \leq \big(\sqrt{- F(\theta)}\big)^2=- F(\theta)$, this would contradict $F(1)>0$.

Now, the function $\gamma(\alpha,\cdot)$ being decreasing and bounded from below by $c^{*}$, it converges to some limit $C^{\infty}$
as $\beta \nearrow 1$. As $\lambda (\alpha,\cdot)$ is increasing, if it does not diverge to $+\infty$ then it converges to some limit
$\lambda^{\infty}$.
We could thus derive a solution $p$ of
$$
\left\{\begin{array}{l}
-p'' - C^{\infty} p' =  f(p),  \hbox{ on } (-\lambda^{\infty},0),\\
\frac{1}{2} (p'(-\lambda^{\infty}))^2 +  F(1) =  F(1),  \\
\frac{1}{2} (p'(\lambda^{\infty}))^2 +  F(\alpha) = 0.
\end{array}\right.
$$
This implies $p'(-\lambda^{\infty})=0$ and thus $p\equiv 1$ by uniqueness, which  contradicts
$\frac{1}{2} (p'(0))^2 +  F(\alpha) = 0$.

The convergence of $\lambda (\cdot, \beta)$ when $\alpha \to 0$ is proved similarly.

Finally, we know that $w_{\alpha, \beta} (p) - F(p) \geq - \min_{[\alpha, \beta]} F$, since $w_{\alpha, \beta} \geq 0$.
Hence if $\beta$ is close enough to $\alpha$, $w_{\alpha, \beta} (p) - F(p) \geq - \f{1}{2}F(\alpha)$ (uniformly in $\beta$).
Then,
\[
2 \lambda(\alpha, \beta) = \int_{\alpha}^{\beta} \f{dp}{\sqrt{2 w_{\alpha, \beta} (p) - F(p) }} \leq \f{\beta -
\alpha}{\sqrt{-F(\alpha)}}
\]
As $\beta \to \alpha$, we deduce that $\lambda(\alpha, \beta) \to 0$, and similarly when $\alpha \to \beta \in (0, \theta_c)$.
\end{proof}

\begin{lemma}\label{lem:lambdaC}
For all $\alpha \in (0, \theta_c)$, $\beta \in (\alpha, 1)$,
\beq
2 \lambda(\alpha, \beta) \gamma(\alpha, \beta) \geq 1 - \sqrt{\f{-F(\theta)}{F(1) - F(\theta)}}.
\label{eq:lambdaC}
\eeq
Moreover, for $0<\beta<\theta_c$, we have 
\beq\label{lim:lambdaC}
\lim_{\alpha \to \beta_-} 2 \lambda(\alpha, \beta) \gamma(\alpha, \beta) = \frac{1}{2} \ln\left(1-\frac{F(1)}{F(\beta)}\right).
\eeq
\end{lemma}
\begin{proof}
The estimate from below is only based on the following inequalities
\[
 F(1) \geq w_{\alpha, \beta} (p) \geq F(p) \geq F(\theta).
\]
They imply, as stated before (in the proof of Proposition \ref{prop:gllimits}):
\[
 \gamma(\alpha, \beta) \geq \f{\sqrt{2}}{\beta - \alpha} \left(\sqrt{F(1) - F(\theta)}- \sqrt{-F(\theta)}\right).
\]
Moreover, $\sqrt{w_{\alpha, \beta}(p) - F(p)} \leq \sqrt{F(1) - F(\theta)}$.
Thus, 
\beq\label{lambda1}
 2 \lambda(\alpha, \beta) \geq (\beta - \alpha) \f{1}{\sqrt{2 \big(F(1) - F(\theta)\big)}}.
\eeq
Combining these estimates yields \eqref{eq:lambdaC}.

Let us fix $\beta\in(0,\theta_c)$, for $0<\alpha<\beta$, we have, 
using \eqref{eq:lambda} and \eqref{eq:w2},
$$
2 \lambda(\alpha, \beta) \gamma(\alpha,\beta) = \int_\alpha^\beta \frac{w'(x)}{2(w(x)-F(x))}\,dx.
$$
On the one hand, we have
\[
\int_\alpha^\beta \frac{w'(x)}{2(w(x)-F(x))}\,dx-\int_\alpha^\beta \frac{w'(x)}{2(w(x)-F(\beta))}\,dx
= \int_\alpha^\beta \frac{w'(x)}{2}\frac{F(x)-F(\beta)}{(w(x)-F(x))(w(x)-F(\beta))}\,dx.
\]
For any $0<\alpha<\beta<\theta_c$, we have $0\leq w(x)\leq F(1)$ then 
$$
\frac{|F(x)-F(\beta)|}{(w(x)-F(x))(w(x)-F(\beta))}\leq \frac{|F(x)-F(\beta)|}{F(x)F(\beta)} 
\leq \left|\frac{1}{F(\beta)}-\frac{1}{F(x)} \right|.
$$
Then, for $\alpha$ close enough to $\beta$, we have 
\begin{align*}
\left|\int_\alpha^\beta \frac{w'(x)}{2}\frac{F(x)-F(\beta)}{(w(x)-F(x))(w(x)-F(\beta))}\,dx\right| 
& \leq \int_\alpha^\beta \frac{w'(x)}{2}\,dx \left|\frac{1}{F(\beta)}-\frac{1}{F(\alpha)} \right|  \\
& = \frac{F(1)}{2} \left|\frac{1}{F(\beta)}-\frac{1}{F(\alpha)} \right|.
\end{align*}
We deduce that 
$$
\int_\alpha^\beta \frac{w'(x)}{2(w(x)-F(x))}\,dx-\int_\alpha^\beta \frac{w'(x)}{2(w(x)-F(\beta))}\,dx 
\to 0, \qquad \mbox{ as } \alpha \to \beta_-.
$$

On the other hand, we compute
$$
\int_\alpha^\beta \frac{w'(x)}{2(w(x)-F(\beta))}\,dx = \frac{1}{2} \ln\left(1-\frac{F(1)}{F(\beta)}\right).
$$
Combining these last identities allows to recover \eqref{lim:lambdaC}.
\end{proof}

\begin{proposition}
 For all $\epsilon > 0$ small enough, there exists $\alpha_{\epsilon} < \beta_{\epsilon}$ with
 \[
  \gamma(\alpha_{\epsilon}, \beta_{\epsilon}) = c_*(f) + \epsilon,
 \]
  and $\alpha_{\epsilon} \to 0$, $\beta_{\epsilon} \to 1$ as $\epsilon \to 0$.
  Moreover,
  $
   \lambda (\alpha_{\epsilon}, \beta_{\epsilon}) \xrightarrow{\epsilon \to 0} +\infty.
  $
 \label{prop:glleps}
\end{proposition}
\begin{proof}
The limit of $\gamma(\alpha, \beta)$ as $\alpha \to 0$ and $\beta \to 1$ exists because
of the monotonicity properties of Proposition \ref{prop:glmonotonicity}.
Moreover, $\gamma(\alpha, \beta)$ is bounded from below by $c_*(f)$.
Simultaneously, we know that $\lambda (\alpha, \beta) \to +\infty$ as $\alpha \to 0$ and $\beta \to 1$ by Proposition
\ref{prop:gllimits}.

The uniqueness of the bistable traveling wave and continuity of $\gamma$ (Lemma~\ref{lem:continuitygl}) imply that
\[
 \lim_{\alpha \to 0, \beta \to 1} \gamma(\alpha, \beta) = c_*(f).
\]
Indeed, let $c$ be this limit. At the limit ($w_{\alpha, \beta}$ and its derivative being uniformly bounded), we get a solution of
\begin{equation*}
\bepa
w' = c \sqrt{2 (w - F)}
\\[10pt]
w(0) = 0, \, w(1) = F(1).
\eepa
\end{equation*}
This exists if and only if $c = c_*(f)$, by uniqueness of the traveling wave solution to the bistable reaction-diffusion equation.
These facts imply the existence of $\alpha_{\epsilon}, \beta_{\epsilon}$.
\end{proof}

The following fact may be proved using Lemma \ref{keyineq}, but also enjoys a simple proof using the properties of $\gamma$, which we propose below.
\begin{proposition}
If $\gamma(\alpha_1, \beta_1) = \gamma(\alpha_2, \beta_2)$, then $\alpha_1 < \alpha_2$ if and only if $\beta_1 < \beta_2$.
\label{prop:aborder}
\end{proposition}
\begin{proof}
Let $C = \gamma(\alpha_1, \beta_1) = \gamma(\alpha_2, \beta_2)$. Assume $\alpha_1 < \alpha_2$.
We can compare $w_1 := w_{\alpha_1, \beta_1}$ and $w_2 := w_{\alpha_2, \beta_2}$ because
$w_2 (\alpha_2) = 0 < w_1 (\alpha_1)$ and as long as $w_2 < w_1$ we also get $w'_2 < w'_1$.
Hence $w_1 (\beta_1) - w_2 (\beta_1) > w_1 (\alpha_1)$. Since $w_1 (\beta_1) = F(1)$ we get
\[
w_2 (\beta_1) < F(1) - w_1 (\alpha_1) < F(1).
\]
Since $w_2$ is increasing and $w_2 (\beta_2) = F(1)$, this implies $\beta_2 > \beta_1$.
\end{proof}

\subsection{Advanced properties of the barrier set.}
\label{subs:advanced}

At this stage, we are ready to prove the following description of $\calB(f)$, which encompasses Theorem \ref{thm:mainHet} and first point of Proposition \ref{prop:Lstarasympt}.
\begin{proposition}
 For all $L > 0$, there exists $C_*(L) > c_* (f)$ such that $(C, L) \in \calB(f) \iff C \geq C_*(L)$.
 For all $C > c_* (f)$, there exists $L_* (C) > 0$ such that $(C, L) \in \calB(f) \iff L \geq L_* (C)$.
 
Furthermore, $C_*(L_*(C)) = C$ and $L_* (C_* (L)) = L$.
\label{prop:strmonotonicity}
\end{proposition}
\begin{proof}
By Propositions \ref{prop:glmonotonicity} and \ref{prop:gllimits}, for any $\alpha \in ( 0, \theta_c)$ and $L>0$, there exists a
unique $\beta_L (\alpha) > \alpha$
such that $\lambda (\alpha, \beta_L (\alpha)) = L$.
In particular, $\big( \gamma(\alpha, \beta_L(\alpha)), L \big) \in \calB(f)$.

Hence $C_* (L) := \inf \{ C > 0, \, (C, L) \in \calB(f) \}$ is well-defined
and because of Proposition \ref{prop:positive}, if $C > C_*(L)$ then $(C, L) \in \calB(f)$.
Moreover, $C_*(L) > c_* (f)$ by Proposition \ref{prop:cstar}

Let $C > c_* (f)$. Then we claim there exists $\alpha, \beta$ such that $\gamma(\alpha, \beta) = C$.
First, for $\epsilon > 0$ small enough, there exists $\alpha_{\epsilon}$ (close to $0$) and $\beta_{\epsilon}$ (close to $1$)
such that $\gamma(\alpha_{\epsilon}, \beta_{\epsilon}) = c_*(f) + \epsilon$, by Proposition~\ref{prop:glleps}.

Hence we can find $\alpha_0, \beta_0$ such that $\gamma(\alpha_0, \beta_0) < C$.

Then since $\gamma(\alpha_0, \beta) \to + \infty$ as $\beta \searrow \alpha_0$ (Proposition \ref{prop:gllimits})
and $\gamma(\alpha_0, \beta)$ is decreasing in $\beta$ (Proposition \ref{prop:glmonotonicity}),
there exists a unique $\beta_C (\alpha_0)$ such that $\gamma (\alpha_0, \beta_C (\alpha_0) ) = C$.
Like before, $L_* (C) := \inf \{ L > 0, \, (C, L) \in \calB(f) \}$ fulfills all properties.

Let $\epsilon > 0$. By definition there exists $\alpha_{\epsilon}, \beta_{\epsilon}$ such that
\[
 \gamma( \alpha_{\epsilon}, \beta_{\epsilon} ) = C, \quad \lambda (\alpha_{\epsilon}, \beta_{\epsilon}) = L_* (C) + \epsilon.
\]
Up to extraction we pass to the limit $\epsilon \to 0$ (the couple $(\alpha_{\epsilon}, \beta_{\epsilon})$ is in a compact set).
Since $\gamma$ and $\lambda$ are continuous, we get $(C, L_* (C)) \in \calB(f)$, and $(C_*(L), L) \in \calB(f)$ by
a similar argument.

Last point boils down to strict monotonicity of $L_*$.
The solution $(X(t), Y(t))$ of
\[
\bepa
\dot{X} = Y, \quad X(0) = \beta,
\\[10pt]
\dot{Y} = - C Y - f(X), \quad Y(0) = - \sqrt{2 \big( F(1) - F(\beta) \big)}
\eepa
\]
depends smoothly on $C$ and $\beta$, so we write it $\big( X(t; C, \beta), Y(t; C, \beta) \big)$. 
We note that by definition
\[
	L_* (C) = \inf_{\beta \in (0, 1)}\,  \inf_{t >0} \, \big\{t, \quad E\big( X(t; C, \beta), Y(t; C, \beta) \big) = 0 \big\}
\]

We denote by $(X_C, Y_C)$ (resp. $(X_{\beta}, Y_{\beta})$) its derivative with respect to $C$ (resp. $\beta$).

From now on we only consider solutions such that $Y < 0$, $X \in [0, 1]$, truncating in time if necessary.

Using indifferently the notations $E = E\big(X(t;C, \beta), Y(t;C, \beta) \big) = E(t; C, \beta)$ we find
\begin{align}
	\p_C E (t) &= \f{\p E}{ \p C} (t; C, \beta) = Y_C (t) Y(t) + X_C (t) f(X(t)),\label{eq:pCE}\\
	\p_{\beta} E(t) &= \f{\p E}{\p \beta} (t; C, \beta) = Y_{\beta} (t) Y(t) + X_{\beta} (t) f(X(t)). \label{eq:pbetaE}
\end{align}
Let $t_* = L_*(C) =  \inf_{\beta \in (0, 1)} \inf \{ t > 0, \, E(t; C, \beta) = 0 \}$, and assume $\beta_* (C) \in (0, 1)$ realizes this
infimum.
We claim that if $\p_C E (t_* (C); C, \beta_*(C) ) < 0$, then $L_*$ is strictly monotone at $C$.

Indeed, let $t_*, \beta_*$ be minimal such that $E\big( X(t_*), Y(t_*) \big) = 0$ and assume $\p_C E(t_*) < 0$.
For $\epsilon > 0$ small enough, $E(t_*; C+ \epsilon, \beta_*) < 0$ by $\p_C E < 0$. Hence there exists $t'_* < t_*$ such that
$E(t'_*; C + \epsilon, \beta_*) = 0$.
This yields $L_* (C + \epsilon) \leq t'_* < t_* = L_* (C)$, that is strict monotonicity.

To prove $\p_C E < 0$, we notice that $(X_C, Y_C)$ and $(X_{\beta}, Y_{\beta})$
are solutions to affine differential systems, with the same linear parts.
\beq
\bepa
\dot{X}_C = Y_C, \quad X_C (0) = 0,
\\[10pt]
\dot{Y}_C = - C Y_C - Y - X_C f'(X), \quad Y_C (0) = 0,
\eepa
\label{eq:wC}
\eeq
and
\beq
\bepa
\dot{X}_{\beta} = Y_{\beta}, \quad X_{\beta} (0) = 1,
\\[15pt]
\dot{Y}_{\beta} = - C Y_{\beta} - X_{\beta} f'(X), \quad Y_{\beta} (0) = \displaystyle\f{f(\beta)}{\sqrt{2 \big( F(1) - F(\beta)
\big)}}.
\eepa
\label{eq:wbeta}
\eeq
Moreover we notice that $X_{\beta} (t) > 0$ for all $t \geq 0$.
Indeed, because of Lemma \ref{keyineq}, $X$ is monotone with respect to its boundary data, that is $X_{\beta} \geq 0$. Then, it suffices to show that $X_{\beta}$ cannot reach $0$ in finite
time. This is a straightforward application of Cauchy-Lipschitz theorem: indeed, since $X_{\beta} \geq 0$, if $X_{\beta} (t_0) = 0$
for some $t_0 > 0$ then $\dot{X}_{\beta} (t_0) = 0$, hence $Y_{\beta} (t_0) = 0$ and finally $(X_{\beta}, Y_{\beta}) \equiv (0, 0)$
by Cauchy-Lipschitz theorem.

Then, we compute the differential equation satisfied by the Wronskian $w(t) := Y_C X_{\beta} - Y_{\beta} X_C$:
\begin{align*}
 w'(t) &= \dot{Y}_C X_{\beta} - \dot{Y}_{\beta} X_C \\
 &= - C w - Y X_{\beta}.
\end{align*}
Because $Y < 0$ and $X_{\beta} > 0$ we get
\[
 \bepa
  (w' + C w)(t) \geq 0 \, \forall t, \quad (w' + C w) (t = 0) > 0,
  \\[10pt]
  w(0) = 0.
 \eepa
\]
Hence if $t > 0$ then $w(t) > 0$.
We can then compute $w$ at $(t_*, \beta_*)$. 
At this point, necessarily $\p_{\beta} E = 0$ (necessary condition for minimality on $\beta$).
And $w(t_*) > 0$ is equivalent to
\begin{align*}
 Y_C X_{\beta} &> X_C Y_{\beta} \\
 \iff Y_C &> \f{X_C Y_{\beta}}{X_{\beta}} \\
 \iff Y_C Y &< \f{X_C Y_{\beta}}{X_{\beta}} Y \text{ by multiplication by } Y < 0\\
 \iff Y_C Y &< - \f{X_{\beta} f(X) X_C}{X_{\beta}} \text{ by } \eqref{eq:pbetaE}\\
 \iff Y_C Y &< - X_C f(X).
\end{align*}
This last inequality is exactly $\p_C E < 0$, and the proof is complete.
\end{proof}
\begin{remark}
Note that we did not use $E = 0$ to prove $\p_C E < 0$.
Therefore, our proof applies for any $t$: the derivative of $E$ with respect to $C$ is negative at the point where $E$ is minimal
(with respect to the initial data $\beta$).
However, we only use this property when the minimum of $E$ is equal to $0$ for our purpose.
\end{remark}

The proposition below is equivalent to Proposition \ref{prop:Lstarasympt}, thanks to Proposition~\ref{prop:strmonotonicity}. 
\begin{proposition}
 The function $C_*$ is non-increasing and satisfies
 
 \begin{enumerate}
  \item[(i)] $\lim_{L \to \infty} C_*(L) = c_*(f)$,
  \item[(ii)] $C_*(L) \sim \displaystyle\f{1}{4 L} \log \big( 1 - \displaystyle\f{F(1)}{F(\theta)} \big)$ when $L \to 0$.
 \end{enumerate}
 \label{prop:cstarlimite}
\end{proposition}
\begin{proof}

The proof of (ii) is a direct consequence of Lemma \ref{lem:lambdaC}. 
Indeed from estimate \eqref{lambda1} we deduce that $\lambda$ goes to $0$ only if $\beta-\alpha\to 0$.
It can occur only if $\beta<\theta_c$. Then with \eqref{lim:lambdaC}, we deduce that 
when $L\to 0$, we have 
$$
C_{*}(L) \sim \frac{1}{4L} \min_{\beta}\ln \left(1-\frac{F(1)}{F(\beta)}\right) = 
\frac{1}{4L} \ln \left(1-\frac{F(1)}{F(\theta)}\right).
$$

For the point (i), we have by Proposition \ref{prop:glleps} that for all $\epsilon > 0$, there
exists $\alpha_{\epsilon}$ (close to $0$) and $\beta_{\epsilon}$ (close to $1$) such that
\[
 \gamma(\alpha_{\epsilon}, \beta_{\epsilon}) = c_*(f)+ \epsilon.
\]
Simultaneously, $\lambda(\alpha_{\epsilon}, \beta_{\epsilon}) \to + \infty$ as $\epsilon \to 0$.
Thus $\lim_{L\to +\infty} C_*(L) = c_*(f)$.

\end{proof}

We now state two auxiliary facts before getting to the proof of our last main result (remaining parts of Proposition \ref{prop:barriers}):
\begin{proposition}
For all $C \geq c_* (f)$ there exists unique $\alpha_C$ and $\beta_C$ such that the generalized problem \eqref{eq:alphabeta} ({\it
i.e.} we impose that its solutions are of class $\mathcal{C}^1$ and let $L = +\infty$) has
 solutions with $(\alpha, \beta) = (\alpha_C, 1)$ and $(\alpha, \beta) = (0, \beta_C)$. 
  When $C = c_*(f)$ this property holds with $(\alpha, \beta) = (0, 1)$: $\alpha_{c_*(f)} = 0$ and $\beta_{c_*(f)} = 1$ for the (unique) traveling wave.
 
The functions $C \mapsto \alpha_C$ and $C \mapsto \beta_C$ are respectively increasing and decreasing. They converge to $0$ and $1$,
respectively, as $C \to +\infty$
 
 Conversely, for any $\alpha \in [0, \theta_c)$ there exists a unique $C \geq c_* (f)$ such that $\alpha = \alpha_C$.
 For any $\beta \in (0, 1]$, there exists a unique $C \geq c_*(f)$ such that $\beta = \beta_C$.
 \label{prop:acbc}
\end{proposition}
\begin{proof}
First we introduce, for all $\alpha \in (0, \theta_c)$ and $\beta \in (0, 1)$:
\[
	C_{\alpha} := \lim_{\beta \to 1} \gamma(\alpha, \beta), \quad C^{\beta} := \lim_{\alpha \to 0} \gamma(\alpha, \beta).
\]

Let us fix $C>c*(f)$. We are going to show that there exists a unique $\alpha\in (0,\theta_c)$ such that $C_\alpha = C$. To this aim, we notice that $\alpha\mapsto C_\alpha$ is continuous, increasing (from Proposition \ref{prop:glmonotonicity}) and $C_0=c_*(f)$. Then it suffices to prove that $\lim_{\alpha\to \theta_c} C_\alpha = +\infty$.
Once this will be done, defining $\alpha_C$ by $C_{\alpha_C} = C$ will yield the result.

Similarly, we are going show that there exists a unique $\beta \in (0, 1)$ such that $C^{\beta} = C$. Again, we notice that $\beta \mapsto C^{\beta}$ is continuous, decreasing, and $C^1 = c_*(f)$. Then it suffices to prove that $\lim_{\beta \to 0} C^{\beta} = +\infty$.

Let $C_{\theta_c} := \lim_{\alpha \to \theta_c} C_{\alpha}, \quad C^0 := \lim_{\beta \to 0} C^{\beta}.$
We are going to prove $C_{\theta_c} = C^0 = +\infty.$

The claim for $C^0$ is a straightforward consequence of Proposition \ref{prop:gllimits}.
For $C_{\theta_c}$, let us assume by contradiction that $C_{\theta_c} < + \infty$.
In this case we find a solution to
\beq
\bepa
- p'' - C_{\theta_c} p' = f(p) \text{ on } (-\infty, 0)
\\[10pt]
- p'' = f(p) \text{ on } (0, +\infty),
\\[10pt]
p(-\infty) = 1, \, p(+\infty) = 0,
\eepa
\eeq
such that $p(0) = \theta_c$. Multiplying the equation by $p'$ and integrating over $(0, +\infty)$ yields $p'(0) = 0$.
However, this cannot hold because by hypothesis ($f$ is bistable), $f(\theta_c) > 0$,
and then this imposes $p''(0) < 0$: $p$ would reach a local maximum at $0$, which contradicts the fact that is has to decrease on
$(-\infty, 0)$.
(Similarly, Hopf Lemma gives that $p'(0) < 0$, which contradicts $p'(0) = 0$.)
\end{proof}

\begin{remark}
In other words, $\alpha_C$ and $\beta_C$ may be defined respectively as
$\alpha_C = p(0)$ where $p$ is the unique solution of class $\mathcal{C}^1$ of 
\[
\bepa
- p'' - C p' = f(p) \text{ on } (-\infty, 0),
\\[10pt]
- p'' = f(p) \text{ on } (0, +\infty),
\\[10pt]
p(-\infty)=1, \, p(+\infty) = 0, \, p > 0.
\eepa
\]
and as $\beta_C = p(0)$ where $p$ be the unique solution of class $\mathcal{C}^1$ of
\[
\bepa
- p''  = f(p) \text{ on } (-\infty, 0),
\\[10pt]
- p'' - C p' = f(p) \text{ on } (0, +\infty),
\\[10pt]
p(-\infty)=1, \, p(+\infty) = 0, \, p > 0.
\eepa
\]
See \cite{Hamel} for existence and uniqueness of these solutions: the results therein apply directly up to transforming $p(\cdot)$ into $p(-\cdot)$ for the first problem, and into $1-p(\cdot)$ for the second one.
\end{remark}

 \begin{lemma}
  Let $C > c^*(f)$. For all $\beta \in (\beta_C, 1)$, there exists a unique $\alpha_C^+ (\beta) \in (0, \alpha_C)$ such that
  $\gamma(\alpha_C^+ (\beta), \beta) = C.$
  We introduce $L^C (\beta) := \lambda (\alpha_C^+ (\beta), \beta)$.
  
  At the limits, $\alpha_C^+ (\beta_C) = 0$ and $\alpha_C^+ (1) = \alpha_C$.
    In addition,
  \[
   \exists \lim_{\beta \to \beta_C} L^C (\beta) = \lim_{\beta \to 1} L^C (\beta) = + \infty.
  \]
    Hence we can define
  \[
   L_m (C) := \min_{\beta \in (\beta_C, 1)} L^C (\beta).
  \]
    Then, $L_m$ is decreasing and $\lim_{C \to + \infty} L_m(C) = 0$.
  \label{lem:claim}
 \end{lemma}
\begin{proof}
Existence and uniqueness for $\alpha_C^+$ (whence the definition of $L^C$) comes from the fact that the equation's flow is strictly inward on the level sets of $E$ (by Lemma \ref{lem:inward}).

The two limits at $\beta_C$ and $1$ of $\alpha_C^+$ are straightforward, as well as those of $L^C$
(this may be seen as a corollary of Proposition \ref{prop:acbc}).
This justifies the existence of a minimum for $L^C$.

Everything being monotone with respect to $C$, this implies that $L_m$ is decreasing.
Finally, the minimality of $L_m$ implies that $L_m \to 0$ as $C \to +\infty$, because (by Proposition \ref{prop:strmonotonicity})
for all $L> 0$, there exists $C_* (L)$, $(C_* (L), L) \in \calB(f)$.
Hence, for $C \geq C_* (L)$, necessarily $L_m(C) < L$.

\end{proof}

We end this subsection by stating and proving an auxiliary fact on the ``limit'' barrier (with minimal length, equal to $L_*(C)$, at a fixed logarithmic gradient $C$).
This fact is not directly useful for proving results of Section \ref{sec:results} but receives a relevant interpretation for the biological problem in Appendix \ref{sec:loc}.
\begin{lemma}
Let $C > c_*(f)$. Let $\alpha_* (C), \beta_*(C)$ be such that
\[
	\gamma \big(\alpha_* (C), \beta_*(C) \big)  = C, \quad 2 \lambda \big(\alpha_*(C), \beta_*(C) \big) = L_*(C).
\]

Then $\alpha_*$ and $\beta_*$ have a limit as $C \to +\infty$, and
\[
 \lim_{C \to \infty} \alpha_* (C) = \theta = \lim_{C \to \infty } \beta_* (C).
\]

In addition, for all $C > c_*(f)$, $\alpha_*(C) < \theta < \beta_* (C)$, and
\[
 \beta_* (C) - \alpha_*(C) = \f{1}{C} \big(\sqrt{2(F(1)-F(\theta))} - \sqrt{-2F(\theta)} \big) + o(\f{1}{C}).
\]

\label{lem:acbctheta}
\end{lemma}

\begin{proof}
For $C > c_*(f)$, there exists $p = p_*^C$ a solution (recall that it is not necessarily unique) of
\[
 \bepa
  - p '' - C p ' = f(p),
  \\[10pt]
  \f{1}{2} p'(-L_*(C))^2 + F(p(-L_*(C))) = F(1), \quad \f{1}{2} p'(L_*(C))^2 + F(p(L_*(C))) = 0.
 \eepa
\]
such that
\[
p_*^C \big( L_*(C) \big) = \alpha_* (C), \quad p_*^C \big( -L_*(C) \big) = \beta_*(C).
\]

We define $v_C : [-1, 1] \to [0, 1]$ by $v_C (x) = p_*^C (x L_*(C))$.
Then $v_C$ satisfies
\[
 \bepa
 -v_C '' - C L_* (C) v_C ' = \big( L_* (C) \big)^2 f(v_C)
 \\[10pt]
 \displaystyle\f{1}{2 \big(L_* (C) \big)^2} v_C' (-1)^2 + F(v_C(-1)) = F(1), \quad \displaystyle\f{1}{2 \big(L_* (C) \big)^2} v_C'(1)^2 + F(v_C(1)) = 0.
 \eepa
\]

We introduce $y = v_C'$.
Recalling that $C L_* (C) \sim_{C \to \infty} \f{1}{4} \log(1 - \f{F(1)}{F(\theta)} )$ (by Proposition \ref{prop:cstarlimite}), for all $z \in (-1, 1)$ we find
\[
 y(z) = y(-1) e^{-C L_*(C)(z+1)} + O( \f{1}{C^2}).
\]
It follows that $ v_C (z) = v_C(-1)+ \f{v_C'(-1)}{C L_*(C)}\big( 1 - e^{-C L_*(C)(z+1)} \big) + O( \f{1}{C^2})$.

Hence $ v_C (1) = v_C(-1) + \f{v_C'(-1)}{C L_*(C)}\big( 1 - e^{-2 C L_*(C)} \big) + O( \f{1}{C^2})$
and $ v'_C(1) = v'_C(-1) e^{-2 C L_*(C)} + O(\f{1}{C^2})$.

From this, we deduce
\beq
\bepa
\f{1}{2 \big(L_* (C) \big)^2} v'_C (-1)^2 + F(v_C (-1)) = F(1)
\\[10pt]
\f{1}{2 \big(L_* (C) \big)^2} v'_C (-1)^2 e^{-4 C L_*(C)} + F\Big(v_C(-1) + \f{v_C'(-1)}{C L_*(C)}\big( 1 - e^{-2 C L_*(C)} \big) \Big) = O(\f{1}{C^2})
\eepa
\eeq
Let $z = v_C(-1)$ and $y = v'_C (-1)$. 
The first equation gives $y = O (1 / C)$, so at the limit $C \to \infty$ we find $\lim_{C \to \infty} v_C(-1) =  \lim_{C \to \infty} v_C (1)$: $v_C$ itself converges to a constant $z_{\infty}$.
Using the first equation in the second we find
\[
  (F(1) - F(z))e^{-4 C L_*(C)} + F\big(z + O(\f{1}{C}) \big) = O(\f{1}{C^2}).
\]
 
Recalling that $e^{4 C L_* (C)} \xrightarrow[C \to \infty]{} 1 - \f{F(1)}{F(\theta)}$
we recover as $C \to \infty$
\[
 F(1) - F(z_{\infty}) + (1 - \f{F(1)}{F(\theta)}) F(z_{\infty}) = 0,
\]
that is $ F(1) \big( 1 - \f{F(z_{\infty})}{F(\theta)} \big) = 0,$
or equivalently $ F(z_{\infty}) = F(\theta)$. 

Hence $\lim_{C \to \infty} z = \theta$.
Recalling $z = v_C(-1) = \beta_*(C)$, we find that both $\alpha_* (C)$ and $\beta_* (C)$ converge to $\theta$.

Let us fix $C > c_*(f)$. For all $\alpha \in (0, \alpha_C)$, there exists a unique $\beta(C, \alpha)$ such that $\gamma(\alpha, \beta(C, \alpha)) = C$.
Obviously, $\alpha \mapsto \beta(C, \alpha)$ is increasing. 

Then, we claim that if $\theta \leq \alpha_0 < \alpha_1 < \alpha_C$ then $\lambda(\alpha_0, \beta(C, \alpha_0)) < \lambda(\alpha_1, \beta(C, \alpha_1))$.
Symmetrically, if $\alpha_0 < \alpha_1 < \alpha_C$ are such that $\beta(C,\alpha_1) < \theta$, then $\lambda(\alpha_0, \beta(C, \alpha_0)) > \lambda(\alpha_1, \beta(C, \alpha_1))$. This is a simple consequence of the expression of $\lambda$ and of the fact that $F$ is decreasing on $[0, \theta]$, increasing on $[\theta, 1]$.

Deriving \eqref{eq:w2} with respect to $p$, choosing $\alpha = \alpha_*(C)$ and $\beta = \beta_*(C)$ and integrating between $\alpha$ and $\beta$ yields
\[
 C \sqrt{2 (w -F)(p)} = C \sqrt{-2 F(\alpha)} + C^2 (p - \alpha) - C \int_{\alpha}^p \f{f(p') dp'}{\sqrt{2 (w - F) (p')}}.
\]
From this we get
\beq
 2 C L_* (C) = C \int_{\alpha}^{\beta} \f{dp}{\sqrt{2 (w -F) (p)}} = \int_{\alpha}^{\beta} \f{dp}{p-\alpha + \f{\sqrt{-2 F(\alpha)}}{C} - \f{1}{C} \int_{\alpha}^p \f{f(p')dp'}{\sqrt{2(w-F)(p')}}}.
 \label{eq:auxbstarastar}
\eeq
By Proposition \ref{prop:cstarlimite} we know that $2 C L_* (C) = \f{1}{2} \log\big(\f{F(1) - F(\theta)}{-F(\theta)} \big) + o(1)$ (where the $o$ is taken as $C \to \infty$).
Rewriting the right-hand side of \eqref{eq:auxbstarastar} (recalling that $\beta_* - \alpha_* = o(1)$), we find
\[
 \f{1}{2} \log\big(\f{F(1) - F(\theta)}{-F(\theta)} \big) = \log\big(1 + C \f{\beta - \alpha}{\sqrt{-2 F(\alpha)}} \big) + o(1).
\]
Since $\alpha \to \theta$ as $C \to +\infty$, taking the exponential of both sides we obtain
\[
 (1 + o(1))\sqrt{2 (F(1) - F(\theta))} = \sqrt{-2 F(\theta)} + C (\beta_*(C) - \alpha_*(C)),
\]
and the claim is proved.

\end{proof}

\subsection{Gathering the results on the barrier set.}
\label{subs:gathering}

We can now prove the remaining parts of Proposition \ref{prop:barriers}, concerning order and extremal elements (recalling the first point has been stated in Lemma \ref{lem:decreasingbarrier}).

\begin{proof}[Proposition \ref{prop:barriers}]
First, we know the $\alpha$s and the $\beta$s are in the same order. More precisely, if there are $(C, L)$-barriers
from $\beta_0$ to $\alpha_0$ and from $\beta_1$ to $\alpha_1$, and $\beta_0 < \beta_1$, then
$\alpha_0 < \alpha_1$ by Proposition~\ref{prop:aborder}.
We then crucially use Lemma \ref{keyineq}.

Applying Lemma \ref{keyineq} to two barriers, on $[-L, L]$ (or equivalently on $[0, 2L]$, to fit the notations in \eqref{sys:XY})
yields
the global ordering of all barriers. Barriers obviously satisfy $X > 0$, $Y < 0$, by Lemma \ref{lem:decreasingbarrier}

Now we take $\lambda_+$ associated with maximal $\beta_+ = p_{\lambda_+} (-L)$ and $\alpha_+ =
p_{\lambda_+} (L)$.
For all $\epsilon > 0$ small enough, we construct a subsolution to \eqref{eq:TW} by letting
\beq
\bepa
-p_{\epsilon}'' - C p_{\epsilon}' = f(p_{\epsilon}) \text{ in } (-L, L), \, p_{\epsilon} (-L) = \beta_+ + \epsilon,
\\[10pt]
-p_{\epsilon}'' = f(p_{\epsilon}) \text{ in } \RR - (-L, L),
\\[10pt]
F(p_{\epsilon} (-L)) + \f{1}{2} (p'_{\epsilon} (-L) )^2 = F(1), \, F(p_{\epsilon} (L)) + \f{1}{2} (p'_{\epsilon} (L^+) )^2 = 0
\eepa
\eeq
where $p_{\epsilon}$ is continuous, but $p'_{\epsilon}$ exhibits a jump at $L$.

Then we can prove that $p_{\epsilon} (L) > p_{\lambda_+} (L)$ and the jump has the good sign to provide a sub-solution $p_{\epsilon}'
(L^-) < p_{\epsilon}' (L_+)$, by maximality of $\beta_+$. The second point can be seen easily in the phase plane.
It is in fact a straightforward consequence of the continuity of $\beta \mapsto E(2L; C, \beta)$

Now, it remains to see that $p_{\epsilon} (x) > p_{\lambda_+} (x)$ for all $x \in [-L, L]$, hence for all $x \in \RR$. In fact, this
is a simple consequence of Lemma \ref{keyineq}. One simply has to check that by continuity of the solutions of differential equations with respect to
the initial data, for $\epsilon > 0$ small enough, $p_{\epsilon}$ remains in $(0, 1)$ on $[-L, L]$ and $p_{\epsilon'}$ remains
negative.

The proof is totally similar for the stability from below of $p_{\lambda_-}$ (defined by minimality of $\beta_- = p_{\lambda_-} (-L)$ and $\alpha_- = p_{\lambda_-}(L)$, making use of Lemma \ref{keyineq} again, hence we
don't reproduce it here.

The last point comes from the fact that $\lambda(\alpha_C^+ (\beta), \beta)$, which is defined on $(\beta_C, 1)$, goes to $+\infty$
at $\beta_C$ and at $1$ (Lemma \ref{lem:claim}), hence reaches its minimum (which is necessarily equal to $L_*(C)$) at some $\beta_0
(C) \in (\beta_C, 1)$.
For $L > L_*(C)$, there exists $(\beta_1, \beta_2)$ with $\beta_C < \beta_1 < \beta_0 (C)$ and $\beta_0(C) < \beta_2 < 1$ such that
$\lambda(\alpha_C (\beta_1), \beta_1) = \lambda (\alpha_C (\beta_2), \beta_2) = L$, yielding two distinct barriers defined by
$(\alpha_C (\beta_i), \beta_i)$ for $i \in \{1, 2\}$.
\end{proof}

\begin{remark}
We interpret Proposition \ref{prop:barriers} in terms of asymptotic behavior of solutions so \eqref{eq:onR} thanks to Proposition
\ref{prop:passblock}.
Any initial datum below $p_{\lambda_-}$ will be unable to pass and propagate (the wave it may have ``initiated'' on $(-\infty, -L)$
will be blocked), while any initial datum above $p_{\lambda_+}$ will propagate.
\end{remark}

\begin{remark}
Proposition \ref{prop:barriers} applies in particular when there exists a unique $(C, L)$ barrier 
(which should generically hold when $L = L_*(C)$). In this case, this barrier is simultaneously stable from below and unstable from
above.
As before, either the solution is blocked below this barrier (``stable from below''), or the solution passes the barrier, in which case
it propagates to $+\infty$ (``unstable from above'').
\end{remark}

\subsection{Generalizing the barriers.}
\label{subs:generalizing}

Now we move to the proof of Corollary \ref{cor:nonconstantgradient}.

\begin{remark}
If $Y := \{ C \chi_{[-L, L]}, \, C, L > 0 \}$, then $\calB (f) = \calB_Y (f)$, our notation for the $(C, L)$-barriers set can be seen as a special case with
$X=Y$, (in fact, \eqref{eq:TW} is a special case of \eqref{eq:GSW}).
\end{remark}

First, we note that these ``generalized'' barriers are still decreasing, as long as $\eta$ is.
\begin{lemma}
 For $\eta \in X$, a $\eta$-barrier is necessarily monotone decreasing.
\end{lemma}
\begin{proof}
Let $L > 0$ be such that $Supp(\eta) \subseteq [-L, L]$.

For $x \in (-\infty, -L)$, since $-p'' = f(p)$ we get by multiplication by $p'$ and integration:
\[
	\f{1}{2} (p'(x))^2 + F(p(x)) = F(1).
\]
Hence $p'$ cannot vanish unless $p = 1$, which is impossible.

Now, for $x \in (L, +\infty)$ we get similarly
\[
	\f{1}{2} (p'(x))^2 + F(p(x)) = 0,
\]
so $p'$ can vanish only if $p = 0$ or $p = \theta_c$. As before, $p=0$ is impossible.
We will show that $p(L) < \theta_c$, which is equivalent to $p'(L) \not=0$, and will be done.

For $x \in (-L, L)$, we define $E(x) := \f{1}{2}(p'(x) )^2 + F(p(x))$.
Then
\[
E'(x) = - \eta(x) p'(x)^2 \leq 0,
\]
so $E$ is non-increasing. (Here it is crucial that $\eta \in X \implies \eta \geq 0$.)
In addition, $E(-L) = F(1)$ and $E(L) = 0$.

Let $x_m := \inf \{ x > -L, \, p'(x) = 0\}$ and assume by contradiction $x_m \leq L$.
If $p(x_m) < \theta_c$ then $E(x_m) = 0 + F(p(x_m)) < 0$, which is absurd because $E$ is non-increasing and $E(L) = 0$.
We are left with $p(x_m) \geq \theta_c > \theta$. This implies that $p''(x_m) = -f(p(x_m)) - \eta(x_m) p'(x_m) < 0$.
In this case, $p$ reaches a local maximum at $x_m$, which is absurd because by definition of $x_m$, $p' < 0$ on $(-L ,x_m)$.

Hence $p$ is monotone decreasing.
\end{proof}

\begin{proposition}
 For all $\eta, \eta_1 \in X$, $\eta \in \calB_X (f) \implies \eta + \eta_1 \in \calB_X (f)$.
 
 If $\lambda > 0$ then $\eta \in \calB_X (f)$ is equivalent to $\lambda \eta (\lambda \cdot) \in \calB_X (\lambda^{2} f)$.
 This point enables us to assume $F(1) = 1$ without loss of generality.
 \label{prop:Gbase}
\end{proposition}
\begin{proof}
The last two points are simple: apart from $\eta$ the rest of the problem is translation-invariant;
$q(x) := p(\lambda x)$ satisfies
\[
 - \f{1}{\lambda^2} q'' (x) - \f{1}{\lambda} \eta(\lambda x) q' (x) = f(q(x)) \text{ on } \RR.
\]
Multiplying this equation by $\lambda^2$ yields the result.

The first point however requires a complete proof, which mimics that of Proposition \ref{prop:positive}. Let $p_{\eta}$ be a $\eta$-barrier.
Then
\[
 -p_{\eta} '' - \big( \eta + \eta_1 \big) p'_{\eta} \geq - p_{\eta}'' - \eta p'_{\eta} = f(p_{\eta}).
\]
Hence $p_{\eta}$ is a super-solution to the $(\eta+\eta_1)$-problem.

Simultaneously, as in the proof of Proposition \ref{prop:positive}, the (translated) $\alpha$-bubble
gives a sub-solution to the $(\eta+\eta_1)$-problem which lies below $p_{\eta}$.

By the sub- and super-solution method, this provides a $(\eta+\eta_1)$-barrier.
\end{proof}
Then, Corollary \ref{cor:nonconstantgradient} follows directly from the first point (positivity) in Proposition \ref{prop:Gbase}
and Theorem \ref{thm:mainHet}.

\section{Discussion and extensions}
\label{sec:dis}

\subsection{Summary of the results}

Before discussing the derivation of the models and some extensions of our results, we sum up the content of the article.

On the first hand, thanks to a change of variables, we established a sharp threshold property for equation \eqref{eq:p} in the bistable case and gave a full description of the situation in the KPP case (Theorem \ref{thm:mainInf}).
Therefore in this simple and homogeneous model, when total population is approximated as a function of infection frequency, no stable propagation blocking can occur.
We also described the propagules in this case (Proposition \ref{prop:propagules}).

On the other hand, when the total population is increasing along a line, we characterized the constant logarithmic gradients that create stable blocking fronts (Theorem \ref{thm:mainHet}), and gave a sufficient condition in Corollary \ref{cor:nonconstantgradient} for the non-constant case. 
We stated the asymptotic behavior of solutions in Proposition~\ref{prop:passblock}, when there are no barriers or when initial data can be compared to some of the barriers.
Then, a deeper understanding of the barriers (Proposition \ref{prop:barriers}) and of the barrier set (Proposition \ref{prop:Lstarasympt}) enabled us to describe the important ``unstable front'' associated with stable blocking fronts.
Computing this unstable front in the context of a blocked artificial introduction of {\it Wolbachia}, for example, may help designing future releases of infected mosquitoes in order to clear the propagation hindrance.

The remainder of this section is organized as follows. We explain in Subsection \ref{sec:der} how \eqref{eq:p} and \eqref{eq:pN} are derived from a two-population model, then in Subsection \ref{sec:loc} we discuss the link between the barriers we considered in this paper and the local barrier studied in \cite{BT}, and finally we gather in Subsection~\ref{sec:numconj} some numerical conjectures we were not able to prove so far.

\subsection{Derivation from a two-population model}

\label{sec:der}
Both \eqref{eq:p} and \eqref{eq:pN} may be derived from a single two-population model.

We consider the model for infected and uninfected mosquitoes proposed in \cite{reduction}.
We denote by $n_i$, resp $n_u$, the density of infected, resp. uninfected, mosquitoes.
\begin{align}
\pa_t n_i - \Delta n_i &= (1-s_f) F_u n_i \big(1-\frac{N}{K}\big) - \delta d_u n_i,
\label{eq:ni}  \\
\pa_t n_u - \Delta n_u &= F_u n_u (1-s_h p) \big(1-\frac{N}{K}\big) - d_u n_u.
\label{eq:nu}
\end{align}
The parameters in this system are: $F_u$ fecundity of uninfected mosquitoes,
$s_f\in (0,1)$ is a dimensionless parameter taking into account the fecundity 
reduction for infected mosquitoes ($F_i=(1-s_f) F_u$ is the fecundity for 
infected mosquitoes), $K$ is the environmental capacity, $d_u$ is the death rate,
$d_i=\delta d_u$ is the death rate for infected mosquitoes ($\delta>1$),
$s_h\in (0,1)$ is the cytoplasmic incompatibility parameter.

We introduce the total population $N=n_i+n_u$ and the fraction of infected
mosquitoes $p=\frac{n_i}{n_i+n_u}$. 
After straightforward computations, we obtain the system
\begin{align}
&\pa_t N - \Delta N = N \left( F_u \big(1-\frac{N}{K}\big) \big((1-s_f)p+(1-p)(1-s_hp)\big) - d_u(\delta p + 1 -p) \right),
\label{append:N}  \\
&\pa_t p - \Delta p - 2 \frac{\nabla p \cdot \nabla N}{N}  = p(1-p)
\left(F_u \big(1-\frac{N}{K}\big) (s_h p -s_f) + d_u (1-\delta)\right).
\label{append:p}
\end{align}

We make the assumption of large population and large fecundity (as in~\cite{reduction}) and introduce $\eps\ll 1$, we rewrite \eqref{append:N} as
$$
\pa_t N - \Delta N = N \left( F_u \big(\frac{1}{\eps}-\frac{N}{K}\big) \big((1-s_f)p+(1-p)(1-s_hp)\big) - d_u(\delta p + 1 -p)
\right),
$$
where both $K$ and $F_u$ are replaced by $K / \epsilon$ and $F_u / \epsilon$.
Linking the carrying capacity and the fecundity in this way appeared as a technical assumption to recover a proper limit as the population goes to $+\infty$, as an equation on the infected proportion $p$. 
Bio-ecology of {\it Aedes} mosquitoes gives a quick but relevant justification of this assumption by the process of ``skip oviposition'': the availability of good-quality containers affects the egg-laying behavior of females, inducing more extensive and energy-consuming search when breeding sites are scarce.
This phenomenon has been documented in \cite{Oviposition1} (for {\it Ae. aegypti}) and \cite{Oviposition2} (for {\it Ae. albopictus}), for example.

Setting $n=\frac{1}{\eps}-N$ and assuming moreover that $\frac{1}{K}=1-\eps\sigma_0$, 
the latter equation rewrites
\begin{multline*}
\pa_t n - \Delta n = \big(n-\frac{1}{\eps}\big) 
\Big( F_u \big(n+\sigma_0 - \eps \sigma_0 n \big) \big((1-s_f)p+(1-p)(1-s_hp)\big) \\- d_u(\delta p + 1 -p) \Big).
\end{multline*}
When $\eps\to 0$ we deduce, at least formally that
\begin{equation}\label{eqlimn}
n + \sigma_0 \to h_0(p) := \frac{d_u(\delta p + 1 - p)}{F_u((1-s_h) p + (1-p)(1-s_hp))}.
\end{equation}
Considering the equation for $p$ \eqref{append:p} with the same scaling,
$$
\pa_t p - \Delta p - 2 \frac{\nabla p \cdot \nabla N}{N} = p(1-p)
\left(F_u \big(\frac{1}{\eps}-\frac{N}{K}\big) (s_h p -s_f) + d_u (1-\delta)\right).
$$
Introducing the variable $n$ as above,
$$
\pa_t p - \Delta p - 2 \frac{\nabla p \cdot \nabla N}{N} = p(1-p)
\left(F_u (n+\sigma_0 -\eps \sigma_0 n) (s_h p -s_f) + d_u (1-\delta)\right).
$$
As $\eps\to 0$, with \eqref{eqlimn}
$$
\pa_t p - \Delta p - 2 \frac{\nabla p \cdot \nabla N}{N} = f(p),
$$
and 
$$
\frac{\nabla N}{N} = \frac{\nabla \sigma_0- \nabla h_0(p)}{\frac{1}{\eps}+\sigma_0-h_0(p)}.
$$

Then, if $\sigma_0$ is constant then we recover equation \eqref{eq:p}.
On the other hand, if the variations of $\sigma_0$ are large (of order $1/\epsilon$), then we may neglect $h$ and thus recover equation \eqref{eq:pN}.

\subsection{Critical population jump}

\label{sec:loc}

In this section we make a link with the concept of barrier strength used in \cite{BT}
for local barriers.
First, we define
\begin{definition}
 A \textbf{local barrier} is a jump ({\it i.e.} a discontinuity) in the size of the total population $N$ which is sufficient to block a propagating wave.
\end{definition}

Starting from our $(L, C)$-barriers, we get a local barrier by letting
$L \xrightarrow{} 0$.
Simultaneously, we scale $C$ as $C(\alpha(L), \beta(L); L)$ for some $\alpha(L) < \beta(L)$.
The jump in the total population, from $N_L$ (on the left) to $N_R > N_L$ (on the right) always reads
\[
 N_R = \exp( \int_{-L}^L \f{C}{2} dx) N_L = \exp( L C ) N_L.
\]

The limit equation as $L \to 0$ reads
\beq \label{LBTWeq}
\left\{\begin{array}{ll}
-p'' - \lim_{L \to 0} \big\{ C(\alpha(L), \beta(L); L) \mathds{1}_{-L\leq x \leq L} p' \big\} = f(p) \qquad &\mbox{ on } \RR,  \\
 p(0^-)=\beta_0, \quad p(0^+) = \alpha_0, \\
p(-\infty)=1, \quad  p(+\infty) = 0,
\end{array}\right.
\eeq
where we assumed $\alpha (L) \xrightarrow[L \to 0]{} \alpha_0$, $\beta (L) \xrightarrow[L \to 0]{} \beta_0$.
Now, recall that by \eqref{lambda1}, necessarily $\alpha_0 = \beta_0$.

This means that $N = N_L$ on $(- \infty, 0)$ and $N= N_R$ on $(0, +\infty)$, with
\[
 N_R = e^{K(\alpha_0)} N_L,
\]
where $K(\alpha_0) = \lim_{L \to 0}  L \cdot C(\alpha(L), \beta(L); L)$.
$K$ depends only on $\alpha_0$ indeed: by formula \eqref{lim:lambdaC} in Lemma \ref{lem:lambdaC},
\[
 K(\alpha_0) = \f{1}{4}\log \big( 1 - \f{F(1)}{F(\alpha_0)} \big).
\]
This implies that
\[
 N_R = \Big( 1 - \f{F(1)}{F(\alpha_0)} \Big)^{1/4} N_L.
\]

Equation \eqref{LBTWeq} then rewrites
\beq \label{2LBTWeq}
\left\{\begin{array}{ll}
-p'' + \f{1}{4} \log \big( 1 - \f{F(1)}{F(\alpha_0)} \big) \langle \delta'_0, p \rangle = f(p) \qquad &\mbox{ on } \RR,  \\
 p(0)=\alpha_0,\\
p(-\infty)=1, \quad  p(+\infty) = 0,
\end{array}\right.
\eeq
and the derivation of \eqref{2LBTWeq} is legitimate for $\alpha_0  = \lim_{L \to 0} \alpha (L) = \theta$, by Lemma~ \ref{lem:acbctheta}.

As a consequence, 
\begin{proposition}
The minimal ``jump'' in the total population that can block a wave is:
\[
 N_R = \Big( 1 - \f{F(1)}{F(\theta)} \Big)^{1/4} N_L.
\]
\label{prop:jump}
\end{proposition}

If we understand \cite{BT} correctly, the authors addressed the situation where for \eqref{LBTWeq}, $ p'(0^-) = p'(0^+)$.
In view of our result, it means $F(1)= 0$.
But simultaneously they wanted $p(0^-) \not= p(0^+)$. 
We find that this cannot be obtained by using equation \eqref{eq:pN}.
However, if the reaction term $f$ depends itself on $N$ (as it is expected to do, see Section \ref{sec:der}), then this becomes possible.

A good intuition is that the stronger the population gradient, the smaller the population ``jump'' required for blocking. In the limit of a real, discontinuous jump, we recover the critical value from Proposition \ref{prop:jump}.

We can state this result in more generality using the notations of this paper.
\begin{proposition}
 Let $H(f, K) := \{ C > c_*(f), \, (C, \f{K}{C}) \in \calB(f) \}$.
 There exists a minimal $K_0(f) > 0$ such that if $K > K_0 (f)$ then $H(f, K)$ is non-empty.
\end{proposition}
\begin{proof}
 We remark that $(C, K/C) \in \calB(f)$ if and only if $K \geq C L_* (C)$, by Theorem~\ref{thm:mainHet}.
 
 Let $K_0 = \min_C C L_* (C) > 0$, and $K > K_0$. Then there exists at least one $C(K) > c_*(f)$ such that $C(K) L_*(C(K)) = K$.
 
\end{proof}

 Assuming $C \mapsto C L_*(C)$ is decreasing (as seems to be the case, see Figure \ref{fig:6} above), a stronger result holds, which confirms the above intuition. 
 In this case, $H(f, K)$ is equal to a half-line for any $K > \Big( 1 - \f{F(1)}{F(\theta)} \Big)^{1/4}$, and is empty otherwise. We refer to \cite{strug:these} for further discussion on this topic.

\subsection{Numerical conjectures}
\label{sec:numconj}

About Lemma \ref{lem:acbctheta}, it is a numerical conjecture that for {\it generic} bistable function $f$, $\alpha_*$ is increasing, $\beta_*$ is decreasing, and both are uniquely defined (see Figure~\ref{fig:astarbstar}).

\begin{figure}[h!]
\centering
 \includegraphics[width=.8\textwidth]{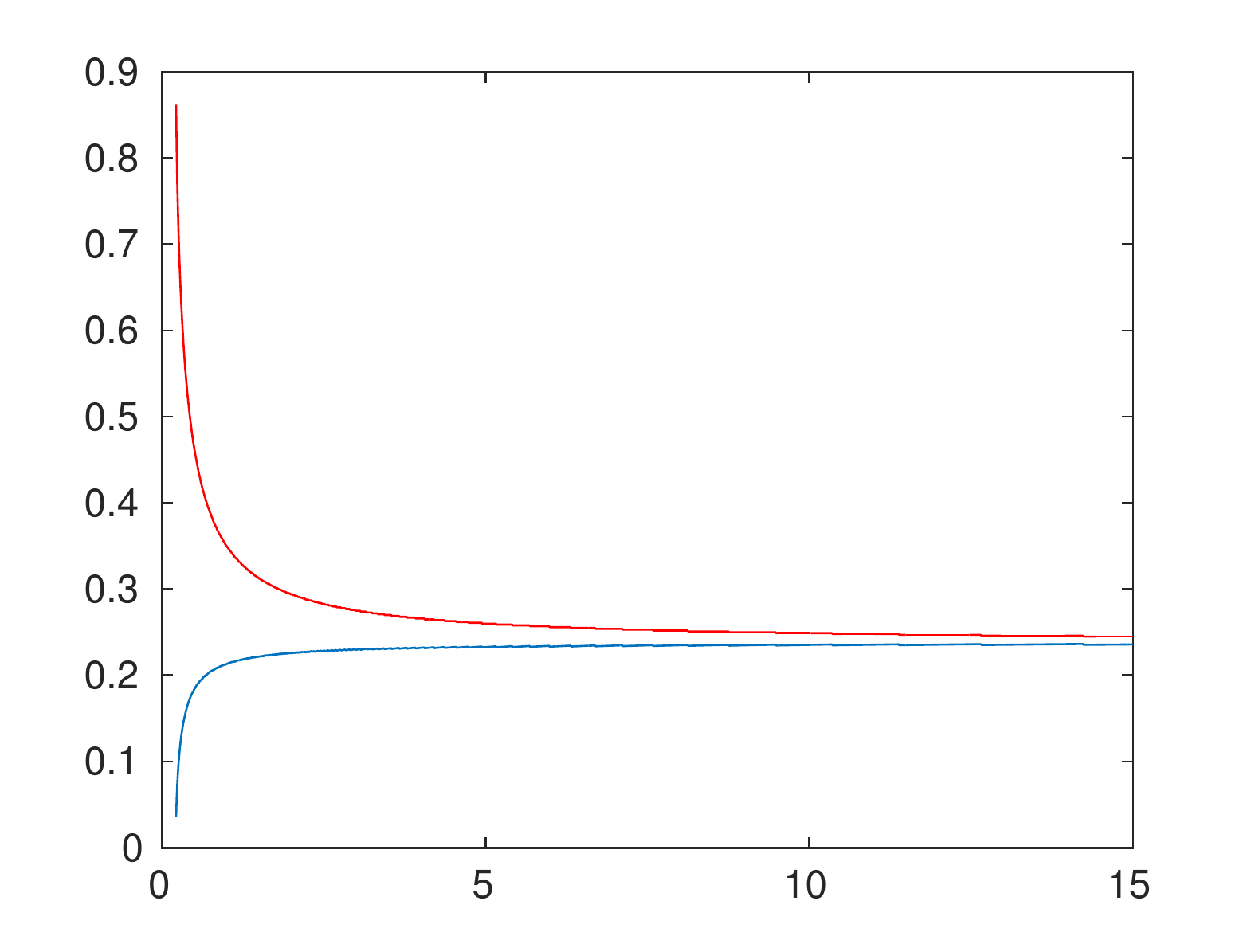}
 \caption{Plot of $\alpha_*$ (in blue, below) and $\beta_*$ (in red, above) as functions of $C$ (respectively increasing and decreasing), obtained by simulating the ODE system \eqref{sys:XY} with $f$ as in Subsection \ref{subs:numerics}.}
 \label{fig:astarbstar}
\end{figure}

For {\it generic} bistable functions, we also conjecture that there exists exactly two barriers when $L > L_*(C)$.

\begin{figure}[h!]
 \includegraphics[width=.5\textwidth]{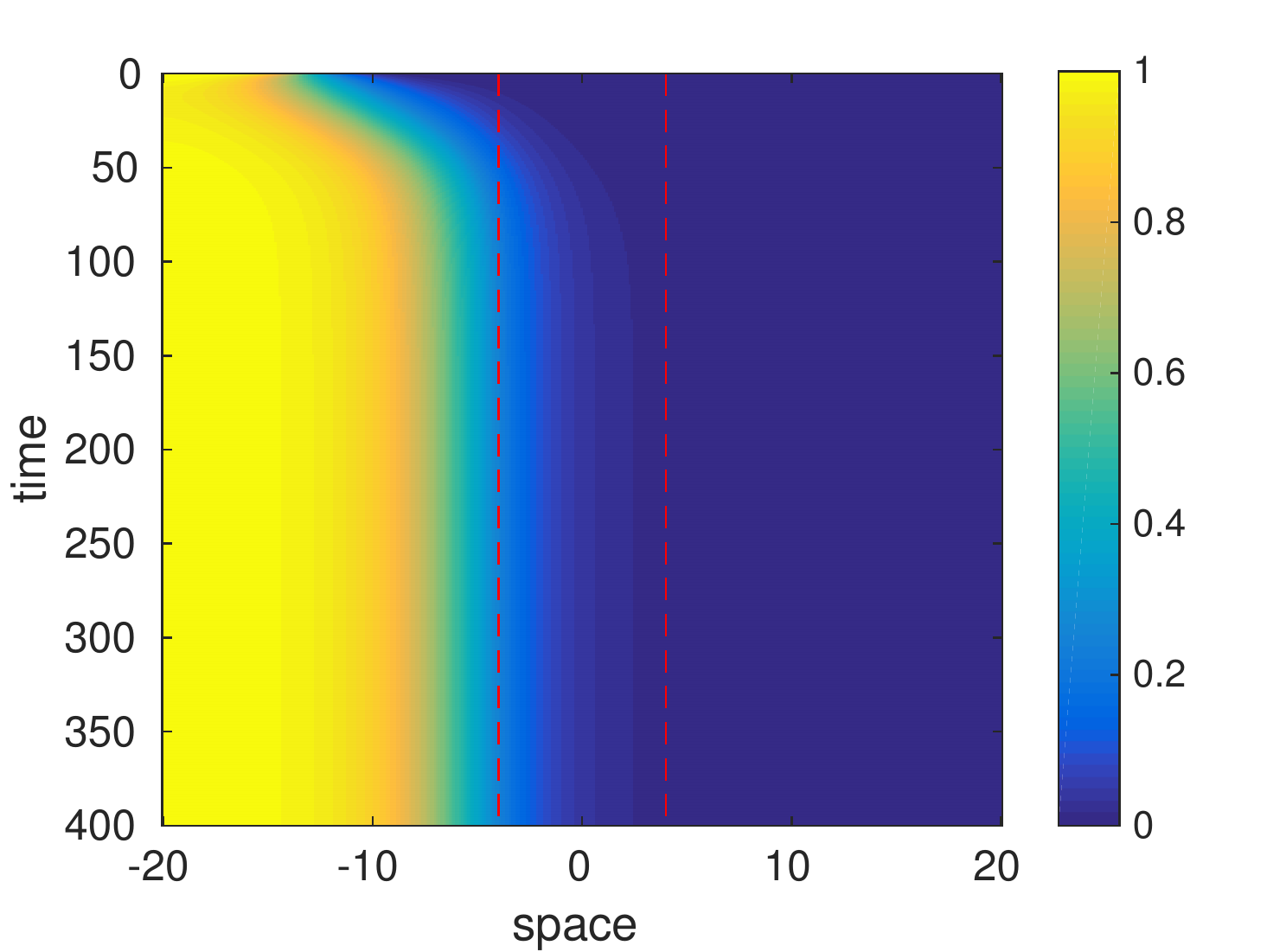}
 \includegraphics[width=.5\textwidth]{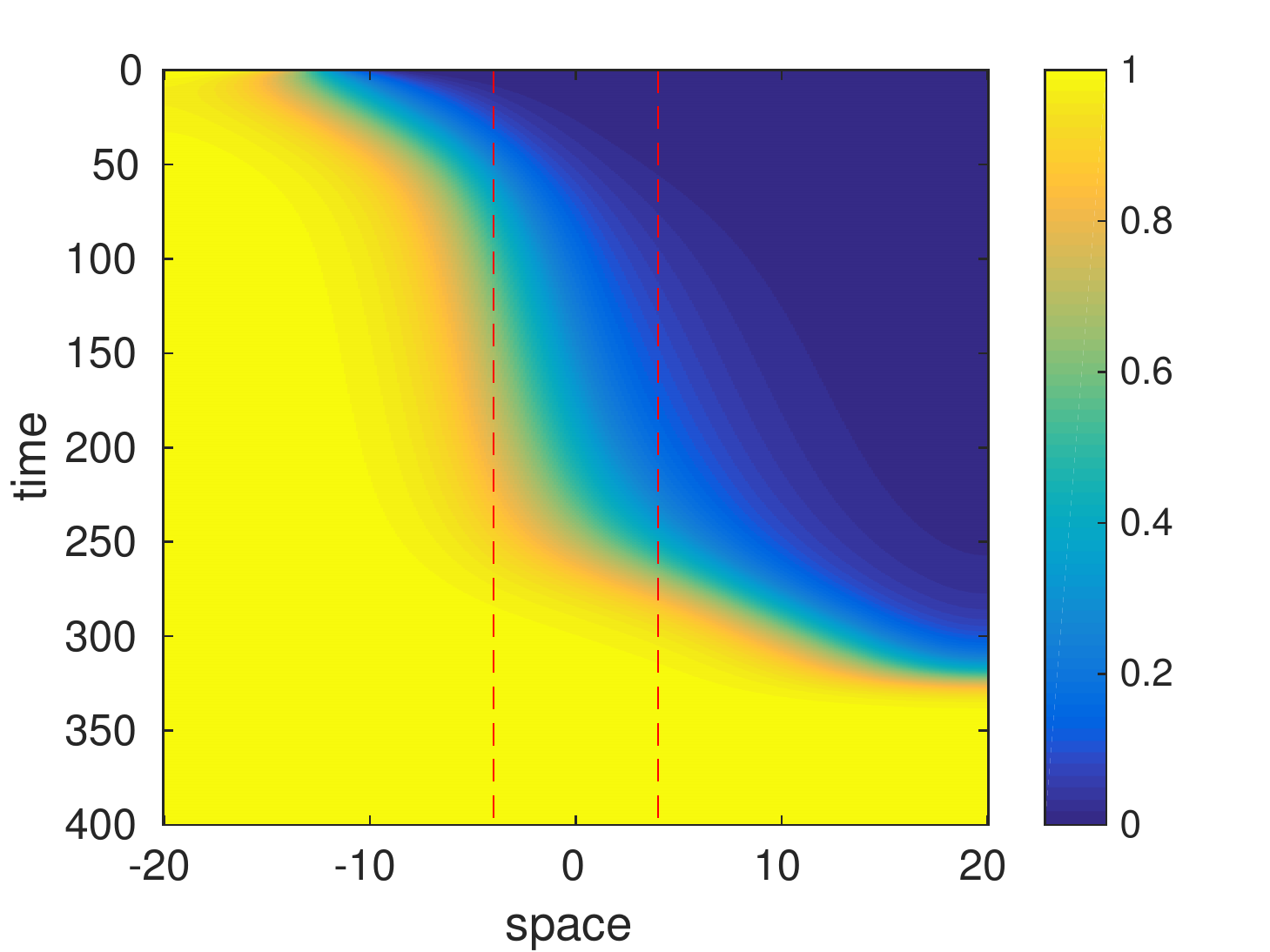}
 \caption{Plot of the proportion of the invading population with respect to time (y-axis) and space (x-axis). These are two numerical simulations of the two-population model \eqref{eq:ni}-\eqref{eq:nu} with same front-like initial data, $L=4$ (space interval with non-zero carrying capacity gradient $[-L, L]$ marked by the two vertical dotted red lines) and two different carrying capacities. We recover the same behavior as for the single population model~\eqref{eq:pN}  {\it Left}: Blocking for $C=0.2$. {\it Right}: Propagation for $C=0.1$.}
 \label{fig:2pop}
\end{figure}

Finally, the behavior we identified appears, numerically, to apply in the case of the two-population model \eqref{eq:ni}-\eqref{eq:nu}, where we take $K = K(x)$ a heterogeneous carrying capacity. Figure \ref{fig:2pop} shows an example of the propagating/blocking alternative in this setting. As in Subsection \ref{subs:numerics}, color represents the value of $p$, which is here equal to $n_i / (n_u + n_i)$.
We fix $L=4$ and choose carrying capacities as
\[
 K(x) = K_L \exp \Big( C \min \big( (x+L)_+, 2L \big) \Big).
\]

%
%
%
%


\newpage
\bibliographystyle{plain}

\end{document}